\documentclass[preprint,12pt]{article}

\usepackage{float}
\usepackage{amssymb}
\usepackage{mathrsfs}
\usepackage{enumerate}
\usepackage{subfigure}
\usepackage{graphicx}
\usepackage{colortbl,dcolumn}
\usepackage{tabularx}
\usepackage{amsmath}
\usepackage{psfrag}
\usepackage{booktabs}
\usepackage{cite}
\usepackage{amsthm}

\numberwithin{equation}{section}
\usepackage[top=1.2in, bottom=1.2in, left=1in, right=1in]{geometry}

\newcommand{\R}{\mathbb{R}}
\newcommand{\N}{\mathbb{N}}

\newtheorem{con}{Condition}[section]
\newtheorem{tm}{Theorem}[section]
\newtheorem{df}{Definition}
\newtheorem{lm}{Lemma}

\newtheorem{rk}{Remark}

\newtheorem{ap}{Assumption}

\begin{document}
\title{High order conformal symplectic and ergodic schemes for stochastic Langevin equation via generating functions}
       \author{
        Jialin Hong\footnotemark[2], Liying Sun\footnotemark[3], and Xu Wang\footnotemark[1]\\
       {\small \footnotemark[2]~\footnotemark[3]~\footnotemark[1] Institute of Computational Mathematics and Scientific/Engineering Computing,}\\{\small Academy of Mathematics and Systems Science, Chinese Academy of Sciences, }\\
         {\small Beijing 100190, P.R.China }}
       \maketitle
       \footnotetext{\footnotemark[2]\footnotemark[3]\footnotemark[1]Authors are supported by National Natural Science Foundation of China (NO. 91130003, NO. 11021101 and NO. 11290142).}
        \footnotetext{\footnotemark[1]Corresponding author: liyingsun@lsec.cc.ac.cn}

       \begin{abstract}
          {\rm\small In this paper, we consider the stochastic Langevin equation with additive noises, which possesses both conformal symplectic geometric structure and ergodicity.
          We propose a methodology of constructing high weak order conformal symplectic schemes by converting the equation into an equivalent autonomous stochastic Hamiltonian system and modifying the associated generating function. To illustrate this approach, we construct a specific second order numerical scheme, and prove that its symplectic form dissipates exponentially. Moreover, for the linear case, the proposed scheme is also shown to inherit the ergodicity of the original system, and the temporal average of the numerical solution is a proper approximation of the ergodic limit over long time.
           Numerical experiments are given to verify these theoretical results.}\\


\end{abstract}

\section{Introduction}
\label{int}

To describe dissipative systems which have interactions with an environment more clearly and specifically, especially in the fields of molecular simulations, quantum systems, cell migrations, chemical interactions, electrical engineering and finances (see \cite{Coffey2012The,gillespie2000the,Schienbein1993Langevin} and references therein), one common way is by means of the stochastic Langevin equation. The stochastic Langevin equation, considered in this paper, is a dissipative Hamiltonian system,  whose phase flow preserves the conformal symplectic geometric structure (\cite{Bourabee2008}) as an extension of the deterministic case. Namely, its symplectic form dissipates exponentially. One can also show that the considered stochastic Langevin equation is ergodic (\cite{MSH02,TD2002,MST10}) with a unique invariant measure, i.e., Boltzmann--Gibbs measure (\cite{Bourabee2008,CPHEVG07}). This dynamical behavior implies that the temporal average of the solution will converge to its spatial average, which is also known as the ergodic limit, with respect to the invariant measure over long time.
 
To approximate the exact solution more accurately and characterize both the geometric structure and the dynamical behavior numerically, this work is developed to propose an approach for constructing high weak order conformal symplectic schemes,
and illustrate this approach by a specific case. We show that the proposed scheme for this particular case inherits the ergodicity of the original system with a unique invariant measure. The weak convergence error, as well as the approximate error of the ergodic limit, is proved to be of order two.

There has been several works concentrating on the construction of numerical schemes for stochastic Langevin equation, mainly based on the splitting technique. For instance, \cite{Bourabee2008} constructs a class of the conformal symplectic integrators to preserve the conformal symplectic structure, and \cite{MilTre2003,MilTre2007} propose the quasi-symplectic methods which can degenerate into symplectic ones when the system degenerates into a stochastic Hamiltonian system. The convergence rate of these schemes depends heavily on the splitting forms.
As for the ergodicity, its numerical analysis essentially follows two directions at our knowledge. The first one is to construct numerical schemes to inherit the ergodicity (see e.g. \cite{MSH02,TD2002}), and gives the error between the numerical invariant measure and the original one (see e.g. \cite{CHW16,BO10}). The other one is to approximate the ergodic limit with respect to the original invariant measure via the numerical temporal averages for some empirical test functions (see e.g. \cite{MST10,MilTre2007,leimkuhler2013IMA}). In the latter case, the numerical solutions may not be ergodic.

In this paper, for the considered stochastic Langevin equation, we aim to construct numerical schemes which are of high weak order and conformal symplectic. To achieve these purposes without bringing the complexity of the high order splitting technique, we introduce a transformation from the stochastic Langevin equation into an autonomous stochastic Hamiltonian system. It then suffices to construct high order symplectic schemes for the autonomous Hamiltonian system, which turns out to be conformal symplectic schemes of the original system based on the inverse transformation of the phase spaces. 
To get high weak order schemes, a powerful tool is the modified equations. For example, \cite{Abdulle} constructs high order stochastic numerical integrators for general stochastic differential equations, but these schemes may not be symplectic when applied to the Hamiltonian systems. Based on the internal properties of the Hamiltonian systems, \cite{Anton1} proposes the method of constructing high weak order stochastic symplectic schemes with multiple stochastic It\^{o} integrals, using truncated generating functions. Based on these schemes, \cite{WHS16} gives their associated modified equations via generating functions.
To reduce the simulation of multiple integrals and still get high weak order symplectic schemes, inspired by \cite{Anton1,WHS16,Abdulle}, we modify the generating function for the equivalent stochastic Hamiltonian system and get associated symplectic numerical methods by truncating modified generating functions. We would like to mention that this kind of methods could avoid simulating too many multiple stochastic It\^{o} integrals, but the products of the increments of Wiener processes instead. This approach is illustrated with the construction of a stochastic numerical scheme which is of weak order two. For the proposed numerical scheme, both the phase volume and symplectic form dissipate exponentially, which coincide with those of the original stochastic Langevin equation. Furthermore, the proposed scheme, similar to the original system, is proved to possess a numerical invariant measure, and the invariant measure is unique for the linear case, which implies the ergodicity of the numerical solution. 
Finally, we verify that both the weak convergence error of the numerical scheme and the error of ergodic limit are of order two. 

An outline of this paper is as follows. Section 2 gives a review of some basic properties of the stochastic Langevin equation, as well as the generating function of the stochastic Hamiltonian system, and also the transformation between the stochastic Langevin equation and an autonomous stochastic Hamiltonian system. In Section 3, a weakly convergent conformal numerical scheme, which possesses an invariant measure, is proposed by means of modified generating functions and the transformation of phase space. In Section 4, we show that both the weak convergence rate of the proposed scheme and the approximate error of the ergodic limit are of order two, based on the uniform estimate of the numerical solutions. At last, we give some numerical tests to verify the theoretical results in Section 5.

\section{Stochastic Langevin equations}
\label{sec2}Let $(\Omega, \mathcal{F}, \mathbb{P})$ be a probability space, $\mathcal{F}_t$ be the filtration for $t\geq 0$ and $W(t)=\big(W_1(t),\cdots, W_m(t)\big)^\top$ be an $m$-dimensional standard Wiener process associated to $\{\mathcal{F}_t\}_{t\geq 0}.$ 
Denote the 2-norm for both matrices and vectors by $\|\cdot\|$ and the determinant of matrices by $\left|\cdot\right|,$ and use $C$ as generic constants independent of $h$ which may be different from line to line.

\subsection{Stochastic conformal symplectic structure and ergodicity}
\label{subsec:ham}
In this section, we focus on stochastic Langevin equation driven by additive noises with deterministic initial values $P(0)=p\in \mathbb{R}^d$ and $Q(0)=q\in \mathbb{R}^d,$ of the following form
\begin{equation}
\begin{split}
\label{1.1}
&dP=-f(Q)dt-vPdt-\sum\limits_{r=1}^{m}\sigma _rdW_r(t),\\
&dQ=MPdt,\quad t\in[0,T],
\end{split}
\end{equation}
where $f\in C^{\infty}(\mathbb{R}^{d},\mathbb{R}^{d}),$ $M\in \mathbb{R}^{d\times d}$ is a positive definite symmetric matrix, $v>0$ is the absorption coefficient and $\sigma_r\in\mathbb{R}^{d}$ with $r\in\{1,\cdots,m\},$ $m\ge d$ and ${\rm rank}\{\sigma_1,\cdots,\sigma_m\}=d$.
In addition, assume that there exists a scalar function $F\in C^{\infty}(\mathbb{R}^{d},\mathbb{R})$ satisfying $$f_i(Q)=\frac{\partial F(Q)}{\partial Q_i},\quad i=1\cdots,d.$$ To simplify the notation, we will remove any mention of the dependence on $\omega\in \Omega$ unless it is absolutely necessary to avoid confusions. Note that (\ref{1.1}) 
holds $\mathbb{P}$-a.s, as well as other stochastic differential equations (SDEs) in the sequel.
It is well known that if $v=0,$ (\ref{1.1}) turns out to be a separable stochastic Hamiltonian system (SHS) which possesses stochastic symplectic structure and phase volume preservation 
(\cite{mil1}). However, 
when $v>0,$ the symplectic form of (\ref{1.1}) dissipates exponentially
\begin{equation*}
dP(t)\wedge dQ(t)=e^{-vt}dp\wedge dq,\quad \forall t\geq 0,
\end{equation*}
which characterizes the longtime tracking of the solutions to (\ref{1.1}), so as the phase volume ${\rm Vol}(t).$
Namely, denote by $D_t=D_t(\omega)\subset \mathbb{R}^{2d}$  
a random domain which has finite volume and is independent of Wiener processes $W(t)$ with respect to the system (\ref{1.1}), one can obtain 
\begin{equation*}
\begin{split}
{\rm Vol}(t)&=\int_{D_t}dP^1\cdots dP^{d}dQ^1\cdots dQ^d\\
&=\int_{D_0}\left|\frac{D(P^1,\cdots,P^{d},Q^1,\cdots,Q^d)}{D(p^1,\cdots,p^{d},q^1,\cdots,q^d)}\right|dp^1\cdots dp^{d}dq^1\cdots dq^d,
\end{split}
\end{equation*}
where the determinant of Jacobian matrix
$\left|\frac{D(P^1,\cdots,P^{d},Q^1,\cdots,Q^d)}{D(p^1,\cdots,p^{d},q^1,\cdots,q^d)}\right|=e^{-vtd}$ with $d$ being the dimension (\cite{mil2,mil1}). 

As another well-known longtime behavior, the ergodicity of \eqref{1.1} is shown in \cite{MSH02} by proving that \eqref{1.1} possesses a unique invariant measure $\mu$. 
Noticing that \eqref{1.1} satisfies the hypoelliptic setting  
\begin{align}\label{hypoelliptic}
{\rm span}\{U_i,[U_0,U_j],i=0,\cdots,m,j=1,\cdots,m\}=\mathbb{R}
^{2d}
\end{align}
with vector fields $U_0=((-f(Q)-vP)^{\top},(MP)^{\top})^{\top}$ and $U_j=(\sigma_j^{\top},0)^{\top}$, $j=1,\cdots,m$, which together with the following assumption yields the ergodicity of \eqref{1.1}.
\begin{ap}
\label{ap1}
Let $F\in C^{\infty}(\R^d,\R)$ satisfy that
\begin{enumerate}[(i)]
  \item $F(q)\ge0$ for all $q\in\R^d$;
  \item there exist $\alpha>0$ and $\beta\in(0,1)$ such that 
  $$\frac12q^{\top}f(q)\ge\beta F(q)+v^2\frac{\beta(2-\beta)}{8(1-\beta)}\|q\|^2-\alpha.$$ 
\end{enumerate}
\end{ap}
Intuitively speaking, the ergodicity of \eqref{1.1} reads that the temporal averages of $P(t)$ and $Q(t)$ starting from different initial values will converge almost everywhere to its spatial average with respect to the invariant measure $\mu.$ More precisely,
\begin{align}\label{ergodiclimit}
\lim_{T\to\infty}\frac1T\int_0^T\mathbf{E}^{(p,q)}\left[\psi(P(t),Q(t))\right]dt=\int_{\mathbb{R}^{2d}}\psi d\mu,\quad\forall~\psi\in C_b(\mathbb{R}^{2d},\mathbb{R})
\end{align}
in $L^2(\mathbb{R}^{2d},\mu)$, where $\mathbf{E}^{(p,q)}[\cdot]$ denotes the expectation starting from $P(0)=p$ and $Q(0)=q$.

Next, we tend to convert \eqref{1.1} into an equivalent homogenous SHS via a transformation of phase space, such that one can construct conformal symplectic schemes for \eqref{1.1} based on symplectic schemes of the homogenous SHS. 
To this end, denoting $X_i(t)= e^{vt}P_i(t)$ and $Y_i(t)=Q_i(t)$ and using It\^o's formula to $X_i(t)$ and $Y_i(t)$ for $i=1,\cdots,d,$ one can rewrite (\ref{1.1}) as
\begin{equation}
\begin{split}
\label{1.1.1}
&dX_i=
-e^{vt}f_i(Y_1,\cdots,Y_d)dt-e^{vt}\sum\limits_{r=1}^m \sigma_rdW_r(t),\quad dY_i=e^{-vt}\sum\limits_{j=1}^dM_{ij}X_jdt
\end{split}
\end{equation}
with $X_i(0)=p_i$ and $Y_i(0)=q_i.$ It is obvious that (\ref{1.1.1}) is a non-autonomous SHS with time-dependent Hamiltonian functions 
\begin{align*}
\tilde H_0=e^{vt}F(Y_1,\cdots,Y_d)+\frac{1}{2}e^{-vt}\sum\limits_{i,j=1}^dX_iM_{ij}X_j,\quad \tilde H_r=e^{vt}\sum\limits_{i=1}^d\sigma_r^i Y_i.
\end{align*}
To obtain an autonomous SHSs we introduce two new variables $X_{d+1}\in \mathbb{R}$ and $Y_{d+1}\in \mathbb{R}$ as the $(d+1)$-th components of $X$ and $Y$, respectively, satisfying
\begin{align*}
dY_{d+1}=dt, \quad dX_{d+1}=-\frac{\partial \tilde H_{0}}{\partial t}dt-\sum\limits_{r=1}^{m}\frac{\partial \tilde H_{r}}{\partial t}\circ dW_r(t)
\end{align*}
with $Y_{d+1}(0)=0$ and $X_{d+1}(0)=F(q_1,\cdots,q_d)+\frac{1}{2}\sum\limits_{i,j=1}^dp_iM_{ij}p_j+\sum\limits_{r=1}^m\sum\limits_{i=1}^d\sigma_r^iq_i.$ Then (\ref{1.1.1}) becomes a $(2d+2)$-dimensional autonomous SHS
\begin{equation}
\begin{split}
\label{1.1.4.1}
&dX=-\frac{\partial H_0}{\partial Y}dt-\sum\limits_{r=1}^m\frac{\partial H_r}{\partial Y}\circ dW_r(t),\quad dY=\frac{\partial H_0}{\partial X}dt+\sum\limits_{r=1}^m\frac{\partial H_r}{\partial X}\circ dW_r(t),
\end{split}
\end{equation}
with $X(0)=(X_1(0),\cdots,X_{d+1}(0))\in \mathbb{R}^{d+1}$, $Y(0)=(Y_1(0),\cdots,Y_{d+1}(0))\in \mathbb{R}^{d+1}$ and new Hamiltonian functions
\begin{align*}
H_0(X,Y)&=e^{vY_{d+1}}F(Y_1,\cdots,Y_d)+\frac{1}{2}e^{-vY_{d+1}}\sum\limits_{i,j=1}^dX_iM_{ij}X_j+X_{d+1},\\
H_r(X,Y)&=e^{vY_{d+1}}\sum\limits_{i=1}^d\sigma_r^i Y_i.
\end{align*}
Here, (\ref{1.1.4.1}) is called the associated autonomous SHS of (\ref{1.1}), and its phase flow preserves the stochastic symplectic structure.  Notice that the motion of the system can be described by different kinds of generating functions (see \cite{Anton1,WH2014} and references therein). We
only consider the first kind of generating function $S$ in this article. 
\subsection{Generating functions} 
\label{subsec:GF}
For convenience, we denote 
$X(0)=x$ and $Y(0)=y$.
It is revealed in \cite{wang} that the generating function $S(X,y,t)$ related to (\ref{1.1.4.1}) is the solution of the following stochastic Hamilton-Jacobi partial differential equation
\begin{align}
\label{1.1.7}
d_{t}S(X,y,t)=H_0(X,y+\frac{\partial S}{\partial X})dt+\sum_{r=1}^{m}H_r(X,y+\frac{\partial S}{\partial X})\circ dW_r(t).
\end{align}
Moreover, 
the mapping $(x,y)\mapsto(X(t),Y(t))$ defined by
\begin{align}
\label{1.1.6}
X(t)=x-\frac{\partial S(X(t),y,t)}{\partial y},\quad Y(t)=y+\frac{\partial S(X(t),y,t)}{\partial X}
\end{align}
is the stochastic flow of (\ref{1.1.4.1}).
Based on It\^o representation theorem and stochastic Taylor-Stratonovich expansion, $S(X,y,t)$ has a series expansion (see e.g. \cite{Anton1,Anton2})
\begin{align}
\label{1.1.8}
S(X,y,t)=\sum_{\alpha}G_{\alpha}(X,y)J^t_{\alpha},
\end{align}
where 
\begin{align*}
J^t_{\alpha}=\int_0^t\int_0^{s_l}\cdots\int_0 ^{s_2}\circ dW_{j_1}(s_1)\circ dW_{j_2}(s_2)\circ\cdots\circ dW_{j_l}(s_l)
\end{align*}
with multi-index $\alpha=(j_1,j_2,\cdots,j_l),$ $j_i\in\{0,1,\cdots,m\},$ $i=1,\cdots,l,$ $l\geq 1$ and $dW_{0}(s):=ds$. To calculate coefficients $G_{\alpha}(X,y)$ in (\ref{1.1.8}), we first give some notations. Let $l(\alpha)$ denote the length of $\alpha$, and $\alpha-$ be the multi-index resulted from discarding the last index of $\alpha$. Define $\alpha\ast \alpha'=(j_1,\cdots,j_l,j_1',\cdots,j_{l'}')$ where $\alpha=(j_1,\cdots,j_l)$ and $\alpha'=(j_1',\cdots,j_{l'}').$ The concatenation `$\ast$' between a set of multi-indices $\Lambda$ and $\alpha$ is $\Lambda\ast\alpha=\{\beta*\alpha|\beta\in \Lambda\}$. Furthermore, define
\begin{equation*}
\Lambda_{ \alpha, \alpha'}=\left\{
\begin{split}
\{(j_1,j_1'),(j_1',j_1)\},\quad &\mbox{if}\quad l=l'=1,\\ 
\{\Lambda _{(j_1),\alpha'- }*(j_{l'}'),\alpha'\ast (j_1)\},\quad &\mbox{if}\quad l=1, l'\neq 1,\\ 
\{\Lambda _{\alpha-,(j_1')}\ast(j_{l}),\alpha\ast (j_{1}')\},\quad &\mbox{if}\quad l\neq1, l'=1,\\ 
\{\Lambda_{\alpha-,\alpha'}\ast (j_l),\Lambda_{\alpha,\alpha'-}\ast (j_{l'}')\},\quad &\mbox{if}\quad l\neq 1, l'\neq 1.\end{split}\right.
\end{equation*}
For $k>2$, let $\Lambda_{\alpha_1,\cdots, \alpha_k}=\{\Lambda_{\beta, \alpha_k}|\beta\in\Lambda_{\alpha_1,\cdots, \alpha_{k-1}}\}.$  We refer to \cite{Anton1} for more details about these notations. Substituting (\ref{1.1.8}) into (\ref{1.1.7}) and taking Taylor expansions to $H_r$ $(r=0,1,\cdots,m)$ at $(X,y)$, we obtain $G_\alpha^1=H_r$ with $\alpha=(r)$ and
\begin{align*}
G_{\alpha}=\sum_{i=1}^{l(\alpha)-1}\frac{1}{i!}
\sum_{k_1,\cdots,k_i=1}^{d+1}\frac{\partial^i H_{j_l}(X,y)}{\partial y_{k_1}\cdots\partial y_{k_i}}
\scriptsize{
\sum_{
\begin{array}{c}
l(\alpha_1)+\cdots+l(\alpha_i)=l(\alpha)-1\\
\alpha-\in\Lambda_{ \alpha_1,\cdots,\alpha_i}
\end{array}{}
}
}
\frac{\partial G_{\alpha_1}}{\partial X_{k_{1}}}\cdots\frac{\partial G_{\alpha_i}}{\partial X_{k_i}}
\end{align*} 
for $\alpha=(j_1,j_2,\cdots,j_l)$ with $l\geq 2$ (see e.g. \cite{Anton1,Anton2}). According to the expression of $G_{\alpha},$ we have $G_{(j_1,j_2)}=\sum\limits_{i=1}^{d+1}\frac{\partial H_{j_2}}{\partial y_i}\frac{\partial H_{j_1}}{\partial X_i}$ and
\begin{align*}
G_{(j_1,j_2,j_3)}=\sum\limits_{i=1}^{d+1}\frac{\partial H_{j_3}}{\partial y_i}\frac{\partial G_{(j_1,j_2)}}{\partial X_i}+\frac{1}{2}\sum\limits_{i,j=1}^{d+1}\frac{\partial^2 H_{j_3}}{\partial y_i\partial y_j}\left(\frac{\partial H_{j_1}}{\partial X_i}\frac{\partial H_{j_2}}{\partial X_j}+\frac{\partial H_{j_2}}{\partial X_i}\frac{\partial H_{j_1}}{\partial X_j}\right).
\end{align*}
Let $C_1:=e^{vy_{d+1}}$ and $C_2:=e^{-vy_{d+1}}.$ Here $y_{d+1}$ denotes the $(d+1)$-th component of $y.$ Note that $y$ is the initial point of the considered interval, that is, if we consider the problem on the interval $[s,t],$ then $y=Y(s).$ For $r_1,r_2,r_3\in\{1,\cdots,m\},$ we have
\begin{equation*}
\begin{split}
\label{2.3.1}
&G_{(r_1,r_2)}=G_{(r_1,0)}=G_{(r_1,r_2,r_3)}=G_{(r_1,r_2,0)}=G_{(r_1,0,r_2)}=0,\\
&G_{(0,r_1)}=\sum\limits_{i,j=1}^{d}\sigma_{r_1}^{i}M_{ij}X_j+vC_1\sum\limits_{i=1}^{d}\sigma_{r_1}^{i}q_{i},\quad G_{(0,r_1,r_2)}=C_1\sigma_{r_1}^{\top} M\sigma_{r_2},\\
&G_{(0,0)}=\sum\limits_{i,j=1}^{d}f_i(y)M_{ij}X_j+vC_1F(y)-\frac{1}{2}vC_2\sum_{i,j=1}^{d}X_iM_{ij}X_j.
\end{split}
\end{equation*}
For a fixed small time step $h,$ using (\ref{1.1.8}) and taking Taylor expansion to $\frac{\partial S}{\partial y_i}:=\frac{\partial S}{\partial y_i}(X,y,h)$ and $\frac{\partial S}{\partial X_i}:=\frac{\partial S}{\partial X_i}(X,y,h)$ at point $(x,y,h)$ for $i=1,\cdots,d,$ we obtain
\begin{align*}
\frac{\partial S}{\partial y_i}=&C_1\left[\sum\limits_{r=1}^{m}\sigma_r^i(J^h_{(r)}+vJ^h_{(0,r)})+f_i(y)\left(h+\frac{vh^2}{2}\right)\right]+\frac{h^2}{2}\sum\limits_{j,k=1}^{d}\frac{\partial^2 F(y)}{\partial y_i\partial y_j}M_{jk}x_k+R_1,\\
\frac{\partial S}{\partial X_i}=&C_2\sum\limits_{j=1}^{d}M_{ij}x_j\left(h-\frac{vh^2}{2}\right)-\sum\limits_{j=1}^{d}\sum\limits_{r=1}^{m}M_{ij}\sigma_r^jJ^h_{(r,0)}-\frac{h^2}{2}\sum\limits_{j=1}^{d}M_{ij}f_j(y)+R_2,
\end{align*}
where every term in $R_1$ and $R_2$ contains the product of multiply stochastic integrals whose the lowest order is at least $\frac{5}{2}$ and so are the remainder terms $R_l$ with $l=3,\cdots,7$ in the sequel.
Furthermore, $\frac{\partial S}{\partial X_{d+1}}(X,y,h)=h$ and
\begin{align*}
\frac{\partial S}{\partial y_{d+1}}=
&vh\bigg{(}C_1F(y)-\frac{C_2}{2}\sum\limits_{i,j=1}^d x_iM_{ij}x_j\bigg{)}\Big{(}1+\frac{vh}{2}\Big{)}
+vC_1\sum\limits_{r=1}^m\sum\limits_{i=1}^d\sigma_r^iy_i(J^h_{(r)}+vJ^h_{(0,r)})\\
&+\sum\limits_{i,j=1}^d\sum\limits_{r=1}^mv\sigma_r^iM_{ij}x_jhJ^h_{(r)}
+vC_1\sum\limits_{r_1,r_2=1}^m\sigma_{r_1}^{\top} M\sigma_{r_2}J^h_{(0,r_1,r_2)}\\
&+v\sum\limits_{i,j=1}^d\left(C_2\frac{\partial F(y)}{\partial y_i}M_{ij}x_jh^2-\frac{1}{2}C_1\sum\limits_{r_1,r_2=1}^m\sigma_{r_1}^iM_{ij}\sigma_{r_2}^jhJ^h_{(r_1)}J^h_{(r_2)}\right)+R_3,
\end{align*}
where $\frac{\partial S}{\partial y_{d+1}}$ takes value at $(X,y,h).$

By truncating the generating function, the weakly convergent stochastic symplectic numerical schemes have been proposed by several authors (see e.g \cite{Anton1,wang,mil1}).
In these approaches, some techniques are applied to simulate the multiple integrals in the truncated generating functions, and obtain high weak order schemes.
To reduce the simulation of multiple integrals, we introduce a modified generating function to construct more concise symplectic schemes in Section \ref{sec3}, such that conformal symplectic and ergodic schemes for stochastic dynamical systems (\ref{1.1}) are deduced by using the transformation of the phase space.

\section{High order conformal symplectic and ergodic schemes}
\label{sec3}
To construct high order symplectic numerical integrators for \eqref{1.1.4.1}, we modify the  stochastic Hamiltonian functions first. Namely,  we consider the following $(2d+2)$-dimensional stochastic Hamiltonian system
\begin{equation}
\begin{split}
\label{3.1}
&dX^{M}=-\frac{\partial H^{M}_0(X^{M},Y^{M})}{\partial Y ^{M}}dt-\sum\limits_{r=1}^{m}\frac{\partial H^{M}_r(X^{M},Y^{M})}{\partial Y ^{M}}\circ dW_r(t),\quad X^{M}(0)=x, \\
&dY^{M}=\frac{\partial H^{M}_0(X^M,Y^M)}{\partial X^M}dt+\sum\limits_{r=1}^{m}\frac{\partial H^{M}_r(X^{M},Y^{M})}{\partial X^M}\circ dW_r(t),\quad Y^{M}(0)=y,
\end{split}
\end{equation}
where 
\begin{equation}
\begin{split}
\label{3.2}
&H^{M}_0(X^{M},Y^{M})={H}_0(X^{M},Y^{M})+{H}_0^{[1]}(X^{M},Y^{M})h+\cdots+{H}_0^{[\tau]}(X^{M},Y^{M})h^{\tau},\\
&H^{M}_r(X^{M},Y^{M})={H}_r(X^{M},Y^{M})+{H}_r^{[1]}(X^{M},Y^{M})h+\cdots+{H}_r^{[\tau]}(X^{M},Y^{M})h^{\tau}
\end{split}
\end{equation}
with functions $H_{i}^{[j]},$ $i=0,\cdots\,r,$ $j=1,\cdots,{\tau},$ $\tau\in\mathbb{N}_{+}$ to be determined. Meanwhile, according to the definition of $G_\alpha$ in Subsection \ref{subsec:GF}, we get  the associated generating function of (\ref{3.1}), which is called the modified generating function of \eqref{1.1.4.1}.
Our goal is to choose undetermined functions in (\ref{3.2}) such that the proposed weakly convergent symplectic numerical approximation is `$k'$ order closer' to the solution of (\ref{1.1.4.1}) than to the solution of (\ref{3.1}).

Now we first give a symplectic numerical approximation to \eqref{3.1} via its generating function, such that this scheme shows weak order $k$ for \eqref{3.1} without specific choices of $H_i^{[j]}$ (see \cite{Anton1} and references therein). In detail, we replace the multiple Stratonovich integrals $J^t_{\alpha}$ in the modified  generating function by an equivalent linear combination of multiple It\^{o} integrals 
\begin{align*}
\label{2.18}
I^t_{\alpha}:=\int_0 ^t\int_0 ^{s_l}\cdots\int_0 ^{s_2}dW_{j_1}(s_1) dW_{j_2}(s_2)\cdots dW_{j_l}(s_l),
\end{align*}
based on the relation
\begin{equation*}
J_{\alpha}^t=\left\{
\begin{split}
&\sum\limits_{\beta}C_\alpha^\beta I_{\beta}^t,\quad l(\alpha)\geq 2,\\
&I_{\alpha}^t,\quad l(\alpha)=1,
\end{split}\right.
\end{equation*}
where $C_\alpha^\beta$ are certain constants which can be found in \cite{kloeden}. 
Denote by 
\begin{align}
\label{SG}
S^{G}(X^G,y,t)=\sum_{\alpha}G^G_{\alpha}( X^G,y)\sum\limits_{l(\beta)\leq k }C_\alpha^\beta I^t_{\beta},
\end{align}
the truncated modified generating function (see e.g. \cite{Anton1,Anton2,kloeden}), where 
\begin{align*}
G_{\alpha}^G=\sum_{i=1}^{l(\alpha)-1}\frac{1}{i!}\sum_{k_1,\cdots,k_i=1}^{d+1}\frac{\partial^i H^{M}_{j_l}(X^G,y)}{\partial y_{k_1}\cdots\partial y_{k_i}}
\scriptsize{
\sum_{
\begin{array}{c}l(\alpha_1)+\cdots+l(\alpha_i)=l(\alpha)-1\\ \alpha-\in\Lambda_{ \alpha_1,\cdots,\alpha_i}
\end{array}
}
}
\frac{\partial G_{\alpha_{1}
}^G}{\partial X_{k_{1}}^G}\cdots\frac{\partial G_{\alpha_i}^G}{\partial X_{k_i}^G}
\end{align*}
for $l(\alpha)\geq 2$, and $G_{(r)}^G=H_r^M$ for $r=0,1,\cdots,m.$ Then we get the following one-step approximation 
\begin{align}
\label{3.5}
X^G=x-\frac{\partial S^G(X^G,y,h)}{\partial y},\quad Y^G=y+\frac{\partial S^G(X^G,y,h)}{\partial X^G},
\end{align}
which preserves symplectic structure and is of weak order $k$ for \eqref{3.1}. Notice that the truncated modified generating function contains undetermined functions $H_{i}^{[j]},$ $i=0,\cdots\,r,$ $j=1,\cdots,{\tau}$ in (\ref{3.2}). To get high weak order symplectic scheme, we need to determine all the $H_{i}^{[j]}$ such that the numerical scheme based on (\ref{3.5}) satisfying
\begin{align}
\label{3.6}
|\mathbf{E}\phi(X(h),Y(h))-\mathbf{E}\phi(X^G,Y^G)|=O(h^{k+k'+1})
\end{align}
for all $\kappa$ times continuously differentiable functions $\phi\in C_P^{\kappa}(\mathbb{R}^{2d+2},\mathbb{R})$ with polynomial growth, that is, the numerical scheme based on (\ref{3.5}) is of weak order $k+k'$ for 
(\ref{1.1.4.1}). 
Conditions on kappa will be given in the following. The detailed approach of choosing the undetermined functions will be illustrated with  the case $k=k'=1$ in next section. We would like to mention that the procedure for constructing conformal symplectic schemes is also available for larger $k$ and $k'.$

\subsection{Numerical schemes via modified generating function}
\label{subsec:MGF}

For $k=k'=1,$ it is sufficient to consider $\tau=1$ in (\ref{3.2}). Based on the fact that $G_{(r)}^G=H_r^M$ for $i=0,1,\cdots,m,$ we rewrite the truncated generating function (\ref{SG}) as
\begin{align}
\label{SG1}
S^{G}(X^G,y,h)=\left(H^{M}_0(X^G,y)+\frac{1}{2}\sum\limits_{r=1}^{m}G^G_{(r,r)}(X^G,y)\right)h+\sum\limits_{r=1}^{m} H^{M}_r(X^G,y)I^h_{(r)},
\end{align}
where 
\begin{align*}
G_{(r,r)}^G
=&%
C_1\sum\limits_{i=1}^d\sigma_r^i\left(\frac{\partial H_r^{[1]}}{\partial X^{G}_i}+vy_i\frac{\partial H_{r}^{[1]}}{\partial X^{G}_{d+1}}\right)h+\sum_{i=1}^{d+1}\frac{\partial H_r^{[1]}}{\partial y_i}\frac{\partial H_r^{[1]}}{\partial X^{G}_i}h^2.
\end{align*} 
According to (\ref{SG1}), the one-step approximation (\ref{3.5}) turns out to be
\begin{equation}
\begin{split}
\label{3.8.2}
&X^{G}=x-\left(\frac{\partial H^{M}_0(X^G,y)}{\partial y}+\frac{1}{2}\sum\limits_{r=1}^{m}\frac{\partial G^G_{(r,r)}(X^G,y)}{\partial y}\right)h-\sum\limits_{r=1}^{m}\frac{\partial H^{M}_r(X^G,y)}{\partial y}J^h_{(r)},\\
&Y^{G}=y+\left(\frac{\partial H^{M}_0(X^G,y)}{\partial X^G}+\frac{1}{2}\sum\limits_{r=1}^{m}\frac{\partial G^G_{(r,r)}(X^G,y)}{\partial X^G}\right)h+\sum\limits_{r=1}^{m}\frac{\partial H^{M}_r(X^G,y)}{\partial X^G}J^h_{(r)}.
\end{split}
\end{equation}
In the sequel, let $\frac{\partial S^G}{\partial y_j}:=\frac{\partial S^G}{\partial y_j}(X^G,y,h),$ $\frac{\partial S^G}{\partial X^G_j}:=\frac{\partial S^G}{\partial X_j^G}(X^G,y,h)$, $\frac{\partial H_r^{[1]}}{\partial y_j}:=\frac{\partial H_r^{[1]}}{\partial y_j}(x,y)$ and $\frac{\partial H_r^{[1]}}{\partial x_j}:=\frac{\partial H_r^{[1]}}{\partial x_j}(x,y)$ for $j=1,\cdots,d+1$ and $r=0,1,\cdots,m.$
Performing Taylor expansion to $\frac{\partial S^G}{\partial y_i}$ and $\frac{\partial S^G}{\partial X^G_i}$ at $(x,y,h)$, for $i=1,\cdots,d,$ we obtain
\begin{equation*}
\begin{split}
\label{3.8.4}
\frac{\partial S^{G}}{\partial X^{G}_i}=&C_2\sum_{j=1}^{d}M_{ij}x_jh+\sum\limits_{r=1}^m\bigg{(}\frac{\partial H_{r}^{[1]}}{\partial x_i}-\sum_{j=1}^{d} M_{ij}\sigma_r^j\bigg{)}I^h_{(r)}h-\sum_{j=1}^{d}M_{ij}f_j(y)h^2+\frac{\partial H_{0}^{[1]}}{\partial x_i}h^2\\
&+\sum\limits_{r=1}^m\frac{\partial^2 H_{r}^{[1]}}{\partial x_i\partial x_{d+1}}(X^{G}_{d+1}-x_{d+1})I^h_{(r)}h-C_1\sum\limits_{r_1,r_2=1}^m\sum_{j=1}^{d}\frac{\partial^2 H_{r_1}^{[1]}}{\partial x_i\partial x_j}\sigma_{r_2}^jI^h_{(r_1)}I^h_{(r_2)}h\\
&+\frac{1}{2}C_1\sum_{j=1}^{d}\sum\limits_{r=1}^m \sigma_r^j\left(\frac{\partial^2 H_{r}^{[1]}}{\partial x_i\partial x_j}+vy_i\frac{\partial^2 H_{r}^{[1]}}{\partial x_i\partial x_{d+1}}\right)h^2+R_4,
\end{split}
\end{equation*}
and
\begin{equation*}
\begin{split}
\label{3.8.5}
\frac{\partial S^{G}}{\partial y_i}=&C_1\sum\limits_{r=1}^{m}\left(\sigma_r^iI^h_{(r)}+f_i(y)h\right)+\sum\limits_{r=1}^{m}\frac{\partial H_r^{[1]}}{\partial y_i}I^h_{(r)}h+\sum\limits_{r=1}^m\sum\limits_{j=1}^{d+1}\frac{\partial^2 H_{r}^{[1]}}{\partial y_i\partial x_j}(X^{G}_{j}-x_j)I^h_{(r)}h\\
+&\left(\frac{\partial H_0^{[1]}}{\partial y_i}+\frac{C_1}{2}\sum\limits_{r=1}^m\left[\sum_{j=1}^{d}\sigma_r^j\frac{\partial^2 H_{r}^{[1]}}{\partial y_i\partial x_j}+v\sigma_r^i\left(\frac{\partial H_{r}^{[1]}}{\partial x_{d+1}}+y_i\frac{\partial^2 H_{r}^{[1]}}{\partial x_{d+1}\partial y_i}\right)\right]\right)h^2+R_5.
\end{split}
\end{equation*}
Similarly, 
\begin{align*}
\frac{\partial S^{G}}{\partial X^{G}_{d+1}}=&h+\sum\limits_{r=1}^{m}\frac{\partial H_{r}^{[1]}}{\partial x_{d+1}}I^h_{(r)}h+\sum\limits_{j=1}^{d+1}\frac{\partial^2 H_{r}^{[1]}}{\partial x_{d+1}\partial x_j}\left(X^{G}_j-x_j\right)I^h_{(r)}h+\frac{\partial H_{0}^{[1]}}{\partial x_{d+1}}h^2\\
&+C_1\sum\limits_{i=1}^d\sigma_r^i\frac{\partial^2 H_r^{[1]}}{\partial x_i\partial x_{d+1}}h^2+C_1\sum\limits_{i=1}^dv\sigma_r^iy_i\frac{\partial^2 H_r^{[1]}}{\partial x_{d+1}^2}h^2+R_6,
\end{align*}
and
\begin{align*}
\frac{\partial S^{G}}{\partial y_{d+1}}
=&v\left(C_1F(y)-\frac{1}{2}C_2\sum\limits_{i,j=1}^d x_iM_{ij}x_j\right)h+vC_1\sum\limits_{r=1}^m\sum\limits_{i=1}^d\sigma_r^iy_iI^h_{(r)}+\sum\limits_{r=1}^m\frac{\partial H_{r}^{[1]}}{\partial y_{d+1}}hI^h_{(r)}\\
&+\sum\limits_{i,j=1}^d\sum\limits_{r=1}^mv\sigma_r^iM_{ij}x_jhI^h_{(r)}+\sum\limits_{r=1}^m\sum\limits_{i=1}^{d+1}\frac{\partial^2 H_{r}^{[1]}}{\partial y_{d+1}\partial x_i}(X^{G}_i-x_i)hI^h_{(r)}+\frac{\partial H_{0}^{[1]}}{\partial y_{d+1}}h^2\\
&+\frac{C_1}{2}\sum\limits_{i=1}^d\sum\limits_{r=1}^m\sigma_r^i\left(v\frac{\partial H_r^{[1]}}{\partial x_i}+v^2y_i\frac{\partial H_{r}^{[1]}}{\partial x_{d+1}}+\frac{\partial^2 H_r^{[1]}}{\partial x_i\partial y_{d+1}}+vy_i\frac{\partial^2 H_{r}^{[1]}}{\partial x_{d+1}\partial y_{d+1}}\right)h^2\\
&+v\sum\limits_{i,j=1}^d\left(C_2\frac{\partial F(y)}{\partial y_i}M_{ij}x_jh^2-\frac{C_1}{2}\sum\limits_{r_1,r_2=1}^m\sigma_{r_1}^iM_{ij}\sigma_{r_2}^jhI^h_{(r_1)}I^h_{(r_2)}\right)+R_7.
\end{align*} 
Applying Taylor expansion to $\phi(X(h),Y(h))$ and $\phi(X^G,Y^G)$ at $(x,y)$ and taking expectations, we have
\begin{equation}
\label{EPD}
\begin{split}
&\mathbf{E}\phi(X(h),Y(h))-\mathbf{E}\phi(X^G,Y^G)\\
=&\sum_{i=1}^{d+1}\frac{\partial \phi(x,y)}{\partial x_i}\mathbf{E}\left(\frac{\partial S^{G}}{\partial y_i}-\frac{\partial S}{\partial y_i}\right)+\sum_{i=1}^{d+1}\frac{\partial \phi(x,y)}{\partial y_i}\mathbf{E}\left(\frac{\partial S}{\partial X_i}-\frac{\partial S^{G}}{\partial X_i^G}\right)\\
&+\frac{1}{2}\sum_{i,j=1}^{d+1}\frac{\partial^2\phi(x,y)}{\partial x_i\partial x_j}\mathbf{E}\left(\frac{\partial S}{\partial y_i}\frac{\partial S}{\partial y_j}-\frac{\partial S^{G}}{\partial y_i}\frac{\partial S^G}{\partial y_j}\right)\\
&+\sum_{i,j=1}^{d+1}\frac{\partial^2\phi(x,y)}{\partial y_i\partial x_j}\mathbf{E}\left(\frac{\partial S^{G}}{\partial X_i^G}\frac{\partial S^{G} }{\partial y_j}-\frac{\partial S}{\partial X_i}\frac{\partial S}{\partial y_j}\right)\\
&+\frac{1}{2}\sum_{i,j=1}^{d+1}\frac{\partial^2\phi(x,y)}{\partial y_i\partial y_j}\mathbf{E}\left(\frac{\partial S}{\partial X_i}\frac{\partial S}{\partial X_j}-\frac{\partial S^{G} }{\partial X_i^G}\frac{\partial S^{G} }{\partial X_j^G}\right)+\cdots.
\end{split}
\end{equation}
To make the symplectic numerical approximation be of higher weak order, we choose $H_i^{[j]},$ $i=0,\cdots,r$, $j=1,\cdots,\tau,$ such that the terms containing $h$ and $h^2$ in the right hand side of (\ref{EPD}) vanish. 
Note that the coefficients of $J_{(r)}^{h}$ and $h$ in $\frac{\partial S^{G}}{\partial X_i^G}$ and $\frac{\partial S^{G}}{\partial y_i}$ 
are the same as those in $\frac{\partial S}{\partial X_i}$ and $\frac{\partial S}{\partial y_i}$ with $i=1,\cdots,d+1,$ respectively.  
Then we get 
\begin{align*}
&\mathbf E\left(\frac{\partial S^{G}}{\partial X^{G}_{d+1}}\frac{\partial S^{G}}{\partial y_{d+1}}-\frac{\partial S}{\partial X_{d+1}}\frac{\partial S}{\partial y_{d+1}}\right)=\sum\limits_{r=1}^m\sum\limits_{i=1}^dvC_1\sigma_{r}^iy_i\frac{\partial H_{r}^{[1]}}{\partial x_{d+1}}h^2+h^3e_1(x,y),
\end{align*}
where $e_1(x,y)$ denotes the coefficient of the term containing $h^3$ and can be calculated based on the expression of the partial derivatives of $S^G$ and $S$, and so are the other remainder terms $e_l,$ $l=2,\cdots,7,$ in the sequel. Thus, we choose
$\frac{\partial H_{r}^{[1]}}{\partial x_{d+1}}=0$ for $r=1,\cdots,m.$
Substituting $\frac{\partial H_{r}^{[1]}}{\partial x_{d+1}}=0$ into $\frac{\partial S^{G}}{\partial X^{G}_{d+1}}$, we have
\begin{align*}
\mathbf{E}\left(\frac{\partial S^{G}}{\partial X^{G}_{d+1}}-\frac{\partial S}{\partial X_{d+1}}\right)=\frac{\partial H_{0}^{[1]}}{\partial x_{d+1}}h^2+\mathbf{E}(R_6)=\frac{\partial H_{0}^{[1]}}{\partial x_{d+1}}h^2+h^3e_2(x,y),
\end{align*}
which lead us to make
$\frac{\partial H_{0}^{[1]}}{\partial x_{d+1}}=0.$
In the same way, using $\frac{\partial H_{r}^{[1]}}{\partial x_{d+1}}=0$ for $r=0,1,\cdots,m,$ we derive
\begin{align*}
\mathbf{E}\bigg{(}\frac{\partial S}{\partial y_{i}}\frac{\partial S}{\partial y_{j}}-\frac{\partial S^{G}}{\partial 
y_{i}}\frac{\partial S^{G}}{\partial y_{j}}\bigg{)}
=C_1\sum\limits_{r=1}^m\bigg{(}vC_1\sigma_r^i\sigma_r^j-\sigma_r^i\frac{\partial H_r^{[1]}}{\partial y_j}-\sigma_r^j\frac{\partial H_r^{[1]}}{\partial y_i}\bigg{)}h^2
+h^3e_3(x,y)
\end{align*}
and
\begin{align*}
\mathbf{E}\left(\frac{\partial S}{\partial y_{i}}\frac{\partial S}{\partial X_{j}}-\frac{\partial S^{G}}{\partial y_{i}}\frac{\partial S^{G}}{\partial X_{j}^G}\right)=C_1\sum\limits_{r=1}^m\sigma_r^i\left(\frac{1}{2}\sum\limits_{k=1}^dM_{jk}\sigma_r^k-\frac{\partial H_r^{[1]}}{\partial x_j}\right)h^2+h^3e_4(x,y)
\end{align*}
with $i,j=1,\cdots,d,$ and hence choose
\begin{equation*}
\frac{\partial H_r^{[1]}}{\partial y_i}=\frac{1}{2}vC_1\sigma_r^i,\quad \frac{\partial H_r^{[1]}}{\partial x_i}=\frac{1}{2}\sum\limits_{j=1}^{d}M_{ij}\sigma_r^j, \quad r=1,\cdots,m.
\end{equation*}
Moreover, because
\begin{align*}
&\mathbf{E}\left(\frac{\partial S}{\partial X_i}\frac{\partial S}{\partial X_j}-\frac{\partial S^G}{\partial X_i^G}\frac{\partial S^G}{\partial X_j^G}\right)=h^3e_5(x,y),\quad i,j=1,\cdots,d+1,
\end{align*}
it has no influence in determining the undetermined functions.
Since both $\frac{\partial H_r^{[1]}}{\partial y_i}$ and $\frac{\partial H_r^{[1]}}{\partial x_i}$ with $r=0,1,\cdots,m,$ are independent of $x_i$ and $y_i,$ it then leads to 
\begin{align*}
&\mathbf{E}\left(\frac{\partial S}{\partial y_{i}}-\frac{\partial S^G}{\partial y_{i}}\right)=\left(\frac{1}{2}\sum\limits_{j,k=1}^{d}\frac{\partial^2 F(y)}{\partial y_i\partial y_j}M_{jk}x_k
+\frac{1}{2}vC_1f_i(y)-\frac{\partial H_0^{[1]}}{\partial y_i}\right)h^2+h^3e_6(x,y),\\
&\mathbf{E}\left(\frac{\partial 
S}{\partial X_{i}}-\frac{\partial S^{G}}{\partial X_{i}^G}\right)=\left(\frac{1}{2}\sum\limits_{j=1}^{d}M_{ij}f_j(y)-\frac{1}{2}\sum\limits_{j=1}^{d}vC_2M_{ij}x_j-\frac{\partial H_0^{[1]}}{\partial x_i}\right)h^2+h^3e_7(x,y)
\end{align*}
for $i=1,\cdots,d.$ We choose $H_0^{[1]}$ such that the above terms containing $h^2$ vanish, i.e.,
\begin{align*}
&\frac{\partial H_0^{[1]}}{\partial y_i}=\frac{1}{2}\sum\limits_{j,k=1}^{d}\frac{\partial^2 F(y)}{\partial y_i\partial y_j}M_{jk}x_k
+\frac{1}{2}vC_1f_i(y),\\
&\frac{\partial H_0^{[1]}}{\partial x_i}=\frac{1}{2}\sum\limits_{j=1}^{d}M_{ij}\left(f_j(y)-vC_2x_j\right).
\end{align*}
Substituting the above results on the partial derivatives of $H_{r}^{[1]},$ $r=0,1,\cdots,m,$ into (\ref{3.8.2}), we have the following scheme of \eqref{3.1}: 
\begin{equation}
\begin{split}
\label{3.8.9}
X^{G}_i=&x_i-\sum\limits_{r=1}^me^{vt_n}\sigma_r^iI^h_{(r)}-e^{vt_n}f_i(y)h-\frac{1}{2}\sum\limits_{r=1}^mve^{vt_n}\sigma_r^ihI^h_{(r)}\\
&-\frac{1}{2}\sum\limits_{j,k=1}^{d}\frac{\partial^2 F(y)}{\partial y_i\partial y_j}M_{jk}X^{G}_kh^2-\frac{1}{2}ve^{vt_n}f_i(y)h^2,\\
Y^{G}_i=&y_i+\sum_{j=1}^{d}e^{-vt_n}M_{ij}X^{G}_jh+\frac{1}{2}\sum\limits_{r=1}^m\sum\limits_{j=1}^{d}M_{ij}\sigma_r^jI^h_{(r)}h\\
&+\frac{1}{2}\sum\limits_{j=1}^{d}M_{ij}\left(f_j(y)-ve^{-vt_n}X^{G}_j\right)h^2,
\end{split}
\end{equation}
which is started at time $t_n=nh$ for $n=1,\cdots,N=T/h.$
That is, $x_i=X_i(t_n)$, $y_i=Y_i(t_n)$ for $i=1,\cdots,d$ and $y_{d+1}=t_n$. 

To transform scheme \eqref{3.8.9} into an equivalent scheme of \eqref{1.1},
we denote $P^h_i[n]:=e^{-vt_n}x_i,$ $Q^h_i[n]:=y_i,$ $P^h_i[n+1]:=e^{-vt_{n+1}}X_i^G$ and $Q^h_i[n+1]:=Y_i^G$ for $i=1,\cdots,d.$ Based on the transformation between two phase spaces of (\ref{1.1}) and (\ref{1.1.4.1}), we get
\begin{equation}
\begin{split}
\label{3.8.11}
P^h[n+1]=&e^{-vh}P^h[n]-\frac{h^2}2\nabla^2F(Q^h[n])MP^h[n+1]-h\Big{(}1+\frac{vh}2\Big{)}e^{-vh}f(Q^h[n])\\
&-\Big{(}1+\frac{vh}2\Big{)}e^{-vh}\sigma\Delta_{n+1}W,\\
Q^h[n+1]=&Q^h[n]+h\left(1-\frac{vh}2\right)e^{vh}MP^h[n+1]+\frac{h^2}2Mf(Q^h[n])+\frac{h}2M\sigma\Delta_{n+1}W,
\end{split}
\end{equation}
where $\sigma=(\sigma_1,\cdots,\sigma_r)$ and $\Delta_{n+1}W=W(t_{n+1})-W(t_{n}).$ Notice that $\Delta_{n}W$ can be simulated by $\xi^n\sqrt{h}$ with $\xi^n=(\xi_1^n,\cdots,\xi_d^n)^{\top}$ being an $\mathcal{F}_{t_n}$-adapted $d$-dimensional normal distributed random vector. 

\begin{rk}
The proposed scheme (\ref{3.8.11}) also has exponentially dissipative phase volume. More precisely, denoting $D(q)=\left(I_d+\frac{h^2}2\nabla^2F(q)M\right)^{-1},$ 
the determinant of Jacobian matrix
\begin{align*}
\left|\begin{array}{cc}{\frac{\partial P^h[1]}{\partial p}}&{\frac{\partial P^h[1]}{\partial q}}\\{\frac{\partial Q ^h[1]}{\partial p}}&{\frac{\partial Q ^h[1]}{\partial q}}\end{array}\right|=&\left|\begin{array}{cc}{e^{-vh}D(q)}&{\frac{\partial P^h[1]}{\partial q}}\\{h(1-\frac{vh}{2})MD(q)}&{D(q)^{-\top}+h(1-\frac{vh}{2})e^{vh}M{\frac{\partial P^h[1]}{\partial q}}}\end{array}\right|\\
=&|e^{-vh}I_{d}||D(q)||D(q)^{-\top}|=e^{-vhd}.
\end{align*}
Furthermore, $\left|\begin{array}{cc}{\frac{\partial P^h[n]}{\partial p}}&{\frac{\partial P^h[n]}{\partial q}}\\{\frac{\partial Q^h[n]}{\partial p}}&{\frac{\partial Q^h[n]}{\partial q}}\end{array}\right|=e^{-vt_nd}.$

\end{rk}

\subsection{Conformal symplectic structure and ergodicity}
In this subsection, we prove the conformal symplecticity of the proposed scheme (\ref{3.8.11}) as well as its ergodicity.
\begin{tm}
\label{th4.1}
The proposed scheme (\ref{3.8.11}) preserves conformal symplectic structure, i.e,
$$dP^h[n+1]\wedge dQ^h[n+1]=e^{-vh}dP^h[n]\wedge dQ^h[n].$$
\end{tm}
\begin{proof}
Based on (\ref{3.8.11}), we obtain
\begin{align*}
&dP^h[n+1]\wedge dQ^h[n+1]\\
=&dP^h[n+1]\wedge dQ^h[n]+\frac{1}{2}h^2dP^h[n+1]\wedge M\nabla^2FdQ^h[n]\\
=&\left(e^{-vh}P^h[n]-\frac{1}{2}h^2\nabla^2FMdP^h[n+1]\right)\wedge dQ^h[n]+\frac{1}{2}h^2dP^h[n+1]\wedge M\nabla^2FdQ^h[n].
\end{align*}
Since the matrix $M$ is symmetric, we finally get
\begin{align*}
dP^h[n+1]\wedge dQ^h[n+1]=e^{-vh}dP^h[n]\wedge dQ^h[n].
\end{align*}
\end{proof}

To show the ergodicity of \eqref{3.8.11}, we give the following conditions which ensure the existence and uniqueness of the invariant measure (see \cite{MSH02} and references therein).
\begin{con}\label{lya}
The Markov chain $Z_n:=(P^h[n]^{\top},Q^h[n]^{\top})^{\top}$ with $Z_0=z$ satisfies:
\begin{enumerate}[(i)]
\item for any $\gamma\ge1$, there exists $C_2=C(\gamma)>0$ which is independent of $h$, such that $\mathbf{E}\|Z_1\|^\gamma\le C_2(1+\|z\|^\gamma)$ for all $z\in\R^{2d}$;
\item there exist $C_1>0$ and $\epsilon>0$ which are independent of $h$, such that $\mathbf{E}\|Z(h)-Z_1\|^2\le C_1(1+\|z\|^2)h^{\epsilon+2}$ for all $z\in\R^{2d}$, where $Z(h)=(P(h)^{\top},Q(h)^{\top})^{\top}.$
\end{enumerate}
\end{con}
\begin{con}\label{min}
For some fixed compact set $G\in\mathcal{B}(\R^{2d})$ with $\mathcal{B}(\R^{2d})$ denoting the Borel $\sigma$-algebra on $\R^{2d}$, the Markov chain $Z_n:=(P^h[n]^\top,Q^h[n]^\top)^\top\in\mathcal{F}_{t_n}$ with transition kernel $\mathcal{P}_n(z,A)$ satisfies:
\begin{enumerate}[(i)]
\item for some $z^*\in\text{\rm int}(G)$ and for any $\delta>0$, there exists a positive integer $n$ such that
$$\mathcal{P}_n(z,B_\delta(z^*))>0,\quad\forall\;y\in G,$$
where $B_\delta(z^*)$ denotes the open ball of radius $\delta$ centered at $z^*$;
\item for any $n\in\N$, the transition kernel $\mathcal{P}_n(z,A)$ possesses a density $\rho_n(z,w)$ which is jointly continuous in $(z,w)\in G\times G$.
\end{enumerate}
\end{con}
\begin{tm}\label{7.3} (see Theorem 7.3 in \cite{MSH02}) For some $K\in\N$, if Condition \ref{lya} and  Condition \ref{min} are satisfied by a Markov chain $Z_n$ when sampled at rate $K$, precisely, these conditions hold for the chain $\tilde Z_n:=Z_{nK}$, then $Z_n$ has a unique invariant measure.

\end{tm}
\begin{tm}\label{ergodic}
Assume that the vector field $f$ is globally Lipschitz. The solution $({P^h[n]},Q^h[n])$ of \eqref{3.8.11}, which is an $\mathcal{F}_{t_n}$-adapted Markov chain, satisfies Condition \ref{lya} and hence admits an invariant measure $\mu_h$ on $\R^{2d}$. In addition, if $f$ is a linear function, then Condition \ref{min} is also satisfied and the invariant measure is unique, that is, \eqref{3.8.11} is ergodic.
\end{tm}
\begin{proof}
{\bf Step 1.} We first show that scheme \eqref{3.8.11} satisfies Condition \ref{lya}. Denote $Z(t)=(P(t)^{\top},Q(t)^{\top})^{\top}\in\R^{2d}$, $Z_n=(P^h[n]^{\top},Q^h[n]^{\top})^{\top}\in\R^{2d}$, $\sigma=(\sigma_1,\cdots,\sigma_r)\in\R^{d\times r}$, $W=(W_1,\cdots,W_r)^{\top}\in\R^r$ and $D(q)=\left(I_d+\frac{h^2}2\nabla^2F(q)M\right)^{-1}$. 
We rewrite \eqref{3.8.11} as
\begin{equation}\label{mild}
\begin{split}
P^h[1]=&D(q)\left(e^{-vh}p-\Big{(}1+\frac{vh}2\Big{)}e^{-vh}\sigma\Delta_1W-h\Big{(}1+\frac{vh}2\Big{)}e^{-vh}f(q)\right),\\
Q^h[1]=&q+h\Big{(}1-\frac{vh}2\Big{)}e^{vh}MP^h[1]+\frac{h^2}2Mf(q)+\frac{h}2M\sigma\Delta_1W
\end{split}
\end{equation}
with $z:=(P_0^{\top},Q_0^{\top})^{\top}=(p^{\top},q^{\top})^{\top}$, 
which yields
\begin{align}\label{X1}
\mathbf{E}\|P^h[1]\|^\gamma+\mathbf{E}\|Q^h[1]\|^\gamma
\le& C(1+\|p\|^\gamma+\|q\|^{\gamma})+C(1+\|q\|^\gamma+\mathbf{E}\|P^h[1]\|^\gamma)\\
\le&C(1+\|p\|^\gamma+\|q\|^\gamma)\nonumber
\end{align}
based on the fact that vector field $f$ is globally Lipschitz, the matrix $I+\frac{h^2}2\nabla^2F(q)M$ is positive definite and $\left\|D(q)\right\|\le1$ for any $q\in\R^d$ and $h\in(0,1)$. 
As the norm $\|Z_1\|=(\|P^h[1]\|^2+\|Q^h[1]\|^2)^{\frac 12}$ is equivalent to the norm $(\|P^h[1]\|^\gamma+\|Q^h[1]\|^\gamma)^{\frac{1}{\gamma}}$, Condition \ref{lya} $(\romannumeral1)$ holds.

Rewrite \eqref{1.1} into the following mild solution form
\begin{equation*}
\begin{split}
&P(h)=p-\int_0^he^{-v(h-s)}f(Q(s))ds-\int_0^he^{-v(h-s)}\sigma dW(s),\\
&Q(h)=q+\int_0^hMP(s)ds
\end{split}
\end{equation*}
with $P(0)=p$ and $Q(0)=q.$ Based on $\eqref{3.8.11}$, we have
\begin{equation*}
\begin{split}
P(h)-P^h[1]=&\bigg{[}h\Big{(}1+\frac{vh}2\Big{)}e^{-vh}f(q)+\frac{h^2}2\nabla^2F(q)MP^h[1]-\int_0^he^{-v(h-s)}f(Q(s))ds\bigg{]}\\
&+\left[\Big{(}1+\frac{vh}2\Big{)}e^{-vh}\sigma\Delta_1W-\int_0^he^{-v(h-s)}\sigma dW(s)\right]\\
=&:\uppercase\expandafter{\romannumeral1}+\uppercase\expandafter{\romannumeral2},\\
Q(h)-Q^h[1]=&\left[\int_0^hMP(s)ds-h(1-\frac{vh}2)e^{vh}MP^h[1]\right]-\left[\frac{h}2M\sigma\Delta_1W+\frac{h^2}2Mf(q)\right]\\
=&:\uppercase\expandafter{\romannumeral3}+\uppercase\expandafter{\romannumeral4}.
\end{split}
\end{equation*}
Now we estimate above terms respectively.
\begin{align}\label{1}
\mathbf{E}\|\uppercase\expandafter{\romannumeral1}\|^2\le&C\mathbf{E}\left\|\frac{h^2}2\nabla^2F(q)P^h[1]\right\|^2+C\mathbf{E}\left\|\int_0^he^{-v(h-s)}\left(f(Q(s))-f(q)\right)ds\right\|^2\nonumber\\
&+C\left\|\int_0^he^{-v(h-s)}dsf(q)-h\Big{(}1+\frac{vh}2\Big{)}e^{-vh}f(q)\right\|^2\nonumber\\
\le&Ch^4(1+\|z\|^2)+C\int_0^he^{-2v(h-s)}ds\int_0^h\left(\|Q(s)-Q^h[1]\|^2+\|Q^h[1]-q\|^2\right)ds\nonumber\\
&+C\left(\frac{1-e^{-vh}}v-h\Big{(}1+\frac{vh}2\Big{)}e^{-vh}\right)^2(1+\|q\|^2)\nonumber\\
\le&Ch^3(1+\|z\|^2)+C\int_0^h\|Q(s)-Q^h[1]\|^2ds,
\end{align}
where in the last step we have used \eqref{X1}. For the term $\uppercase\expandafter{\romannumeral2}$, based on the It\^o isometry,
\begin{align}\label{2}
\mathbf{E}\|\uppercase\expandafter{\romannumeral2}\|^2\le\int_0^h\left(\Big{(}1+\frac{vh}2\Big{)}e^{-vh}-e^{-v(h-s)}\right)^2ds{\rm Tr}\left(\sigma\sigma^{\top}\right)\le Ch^3.
\end{align}
Similarly, 
we have
\begin{align}\label{3}
\mathbf{E}\|\uppercase\expandafter{\romannumeral3}\|^2
\le&C\mathbf{E}\left\|\int_0^hM\left(P(s)-P^h[1]\right)ds\right\|^2+C\mathbf{E}\left\|h\left(1-\Big{(}1-\frac{vh}2\Big{)}e^{vh}\right)MP^h[1]\right\|^2\\
\le&C\int_0^h\|P(s)-P^h[1]\|^2ds+Ch^4(1+\|z\|^2),\nonumber
\end{align}
and
\begin{align}
\label{4}
\mathbf{E}\|\uppercase\expandafter{\romannumeral4}\|^2
\le& Ch^3(1+\|q\|^2).
\end{align}
From \eqref{1}--\eqref{4}, we conclude
\begin{align*}
\mathbf{E}\|Z(h)-Z_1\|^2\le C\int_0^h\mathbf{E}\|Z(s)-Z_1\|^2ds+Ch^3(1+\|z\|^2),
\end{align*}
which together with Gronwall inequality yields Condition \ref{lya} $(\romannumeral2)$ with $\epsilon=1$.
In this case, there exist real numbers $\alpha\in(0,1)$ and $\beta\in[0,\infty)$ such that $\mathbf{E}[V(Z_{n+1})|\mathcal{F}_{t_n}]\le\alpha V(Z_n)+\beta$ for
$V(z)=\frac12\|p\|^2+F(q)+\frac{v}{2}p^\top q+\frac{v^2}{4}\|q\|^2+1$ with $z=(p^\top, q^\top)^\top$ (see Theorem 7.2 \cite{MSH02}). Hence, 
$$\mathbf{E}[V(Z_{n+1})]\le\alpha\mathbf{E}[V(Z_n)]+\beta\le \alpha^{n+1}\mathbf{E}[V(Z_0)]+\beta\frac{1-\alpha^n}{1-\alpha}\le C(Z_0),$$
which induces the existence of invariant measures (see Proposition 7.10 \cite{prato06}).

{\bf Step 2.} We now consider the chain $Z_{2n}$ sampled at rate $K=2$ and verify Condition \ref{min} when $f$ is linear with a constant $C_f:=\nabla f=\nabla^2F$. 
Let $G:=\left\{(P^{\top},Q^{\top})^{\top}\in\R^{2d}:Q=0,\|P\|\le 1\right\}$ which is a compact set. For any $z=(p^{\top},0)^{\top}\in G$ and $w=(w_1^{\top},w_2^{\top})^{\top}\in B$ with $B\in\mathcal{B}(\R^{2d})$, we tend to show that $\Delta_1W$ and $\Delta_2W$ can be properly chosen to ensure that $P^h[2]=w_1$ and $Q^h[2]=w_2$ starting from $(P_0^{\top},Q_0^{\top})^{\top}=z$. Denoting $L_h=h\left(1-\frac{vh}2\right)e^{vh}M$, from \eqref{3.8.11}, we have
\begin{align}\label{y1}
w_1=&e^{-vh}P^h[1]-\frac{h^2}2C_fMw_1-h\Big{(}1+\frac{vh}2\Big{)}e^{-vh}f(Q^h[1])
-\Big{(}1+\frac{vh}2\Big{)}e^{-vh}\sigma\Delta_{2}W,\\\label{y2}
w_2=&Q^h[1]+L_hw_1+\frac{h^2}2Mf(Q^h[1])+\frac{h}2M\sigma\Delta_{2}W,\\
=&Q^h[1]+L_hw_1+\frac{h}2\Big{(}1+\frac{vh}2\Big{)}^{-1}e^{vh}M\left(e^{-vh}P^h[1]-w_1-\frac{h^2}2C_fMw_1\right)\nonumber\\\label{p1}
P^h[1]=&e^{-vh}p-\frac{h^2}2C_fMP^h[1]-h\Big{(}1+\frac{vh}2\Big{)}e^{-vh}f(0)
-\Big{(}1+\frac{vh}2\Big{)}e^{-vh}\sigma\Delta_{1}W,\\\label{q1}
Q^h[1]=&L_hP^h[1]+\frac{h^2}2Mf(0)+\frac{h}2M\sigma\Delta_{1}W\\
=&L_hP^h[1]+\frac{h}2\Big{(}1+\frac{vh}2\Big{)}^{-1}e^{vh}M\left(e^{-vh}p-P^h[1]-\frac{h^2}2C_fMP^h[1]\right)\nonumber.
\end{align}
Noticing that \eqref{y2} and \eqref{q1} form a linear system, from which we can get the solution $P^h[1]$ and $Q^h[1]$ based on the positive definite coefficient matrix. Then $\Delta_2W$ and $\Delta_1W$ can be uniquely determined by \eqref{y1} and \eqref{p1} respectively.
Condition \ref{min} $(\romannumeral1)$ is then ensured according to the property that Brownian motions hit a cylinder set with positive probability. 
For Condition \ref{min} $(\romannumeral2)$, from \eqref{mild}, we can find out that $P^h[1]$ has a $C^{\infty}$ density based on the facts $\Delta_1W$ has a $C^{\infty}$ density, $\sigma$ is full rank and $D(q)$ is positive definite for any $q\in\R^d$. Thus, $Q^h[1]$ also has a $C^{\infty}$ density, and Theorem \ref{7.3} is applied to complete the proof.
\end{proof}

\section{Approximate error}
In this section, we turn to consider the weak convergence order of (\ref{3.8.11}) by investigating the local convergence error first. Furthermore, based on the local convergence error and the hypoelliptic setting \eqref{hypoelliptic}, we can also get the approximate error of the ergodic limit. Denote the exact solution of (\ref{1.1}) and the numerical solution by $Z(t)=(P(t)^\top,Q(t)^\top)^\top$ and $Z_n=(P^h[n]^\top,Q^h[n]^\top)^\top,$ respectively. Next theorem gives that the moments of (\ref{1.1}) are uniformly bounded, whose proof is in the same procedure as Lemma 3.3 in \cite{MSH02}.

\begin{tm}
\label{th4.1.1}
Let Assumption \ref{ap1} holds, then for any $k\in\mathbb{N}_{+},$ the $k$-th moments of $P(t)$ and $Q(t)$ are uniformly bounded with respect to $t\in\mathbb{R}_{+}.$
\end{tm}

Before proving the main convergence theorem, we first show the boundedness of the numerical solution to (\ref{3.8.11}) in the following theorem.
\begin{tm}
\label{th4.1.2}
Assume that the coefficient $f$ of equation (\ref{1.1}) is globally Lipschitz and satisfies the linear growth condition, i.e.,
\begin{equation}
\label{4.13.2}
\|f(q_1)-f(q_2)\|\leq L\|q_1-q_2\|,\quad \|f(\tilde q)\|\leq C_f(1+\|\tilde q\|)
\end{equation}
for some constants $L\geq 0$ and $C_f\geq 0$, and any $q_1,q_2,\tilde q\in \mathbb{R}^d.$ 
Then there exists a positive constant $h_0$ such that for any $h\leq h_0,$ it holds
$$\sup\limits_{n\in\{1,\cdots,N\}}\mathbf{E}\left[\|P^h[n]\|^{k}+\|Q^h[n]\|^{k}\right]<\infty.$$
%
\end{tm}
\begin{proof}
For any fixed initial value $z=(p^\top,q^\top)^\top,$ random variable $\xi:=\xi^1$ and $h,$ we have based on \eqref{3.8.11} that 
\begin{align*}
\|P^h[1]-p\|\leq&|e^{-vh}-1|\|p\|+ h\Big{(}1+\frac{vh}2\Big{)}\|f(q)\|+\sqrt{h}\Big{(}1+\frac{vh}2\Big{)}\|\sigma\xi\|\\
&+\frac{h^2}2\|\nabla^2F(q)\|\|M\|\|p\|+\frac{h^2}2\|\nabla^2F(q)\|\|M\|\|P^h[1]-p\|.
\end{align*}
Denote $C_v:=1+\frac{vh}{2}.$ Using the global Lipschitz condition and mean value theorem, there exists some $\theta\in (0,1)$ such that 
\begin{align*}
\|P^h[1]-p\|\leq& |-vhe^{-v\theta h}|\|p\|+ hC_f(1+\|z\|)+\sqrt hC_v\|\sigma\xi\|\\
&+\frac{h^2}2L\|M\|\|z\|+\frac{h^2}2L\|M\|\|P^h[1]-p\|\\
\leq &C(1+\|z\|)(\|\xi\|\sqrt{h}+h)+L\|M\|\|P^h[1]-p\|\frac{h^2}{2}.
\end{align*}
It is obvious that there exists a positive constant $h_0$ such that for any $h\leq h_0$,$$L\|M\|\frac{h^2}{2}\leq \frac12.$$ 
It then yields  
\begin{align*}
\|P^h[1]-p\|\leq& 2C(1+\|z\|)(\|\xi\|\sqrt{h}+h).
\end{align*}
On the other hand,
for $h\leq h_0,$ we have
\begin{align*}
&\|\mathbf{E}(P^h[1]-p)\|\\
\leq& \left\|(e^{-vh}-1)p-\frac{h^2}2\nabla^2F(q)Mp-hC_ve^{-vh}f(q)\right\|+\left\|\frac{h^2}2\nabla^2F(q)M\mathbf{E}(P^h[1]-p)\right\|\\
\leq&vh\|p\|+hL\|M\|\|p\|+hC_fC_v(1+\|z\|)+\frac{h^2}2L\|M\|\|\mathbf{E}(P^h[1]-p)\|,
\end{align*}
which leads to
\begin{align*}
\|\mathbf{E}(P^h[1]-p)\|\leq C(1+\|z\|)h.
\end{align*}
Based on the estimate of $P^h[1]-p,$ similarly, we have
\begin{align*}
\|Q^h[1]-q\|\leq C(1+\|z\|)(\|\xi\|\sqrt{h}+h),\quad \|\mathbf{E}(Q^h[1]-q)\|
\leq C(1+\|z\|)h.
\end{align*}
We can conclude that, for $Z_1=(P^h[1]^\top,Q^h[1]^\top)^\top$,
\begin{align}
\label{6.1}
\|Z_1-z\|\leq & C(\|\xi\|+\sqrt h)(1+\|z\|)\sqrt h
\leq C(\|\xi\|+1)(1+\|z\|)\sqrt{h}.
\end{align} 
Thus, we complete the proof according to Lemma 9.1 in \cite{Milstein1988}.
\end{proof}

Based on the above preliminaries, our result concerning the weak convergence order of the proposed
scheme is as follows.
\begin{tm}
\label{th4.2}
Under the assumptions in Theorem \ref{th4.1.2}, 
the proposed scheme (\ref{3.8.11}) is of weak order 2. More precisely,
\begin{align*}
\left|\mathbf{E}\psi\left(P(T),Q(T)\right)-\mathbf{E}\psi\left(P^h[N],Q^h[N]\right)\right|=O(h^{2})
\end{align*}
for all $\psi\in C_P^6(\mathbb{R}^{2d},\mathbb{R})$  and $T=Nh.$
\end{tm}

\begin{proof}
Without loss of generality, we consider the case of $d=1.$ 
Based on It\^o's formula, Theorem \ref{th4.1.1} and \ref{th4.1.2}, we obtain
\begin{equation*}
\begin{split}
P(h)=&p-\int_{0 }^{h }\left(f(Q(s))+vP(s)\right)ds-\sum\limits_{r=1}^m\int_{0 }^{h }\sigma_rdW_r(s)\\
=&p-\int_{0 }^{h }\left(f(q)+\int_{0 }^{s}\nabla^2F(Q(\theta))MP(\theta)d\theta\right)ds-\sum\limits_{r=1}^m\int_{0 }^{h}\sigma_rdW_r(s)\\
&-v\int_{0 }^{h }\left(p-\int_{0 }^{s}f(Q(\theta))d\theta-\int_{0 }^{s}vP(\theta)d\theta-\sum\limits_{r=1}^m\sigma_rdW_r(\theta)\right)ds,
\end{split}
\end{equation*}
which leads to 
\begin{equation}
\begin{split}
\label{4.3}
P(h)=&p-f(q)h-vph-\frac{1}{2}\nabla^2F(q)Mph^2-\sum\limits_{r=1}^m\int_{0 }^{h}\sigma_rdW_r(s)\\
&+\frac{1}{2}vf(q)h^2+\frac{1}{2}v^2ph^2+v\sum\limits_{r=1}^m\int_{0 }^{h}\int_{0}^{s}\sigma_rdW_r(\theta)ds+\delta_1,
\end{split}
\end{equation}
where $\mathbf E\|\delta_1\|=O(h^3)$ and $\mathbf{E}\|\delta_1\|^2=O(h^5).$
Analogously, it also holds that
\begin{equation}
\begin{split}
\label{4.4}
Q(h)=&q+\int_{0 }^{h }M\left(p-\int_{0 }^sf(Q(\theta))d\theta-v\int_{0 }^sP(\theta)d\theta-\sum\limits_{r=1}^m\int_{0 }^s\sigma_rdW_r(\theta)\right)ds\\
=&q+Mph-\frac{1}{2}f(q)h^2-\frac{1}{2}vMph^2-\sum\limits_{r=1}^mM\sigma_r\int_{0}^{h}\int_{0}^{s}dW_r(\theta)ds+\delta_2
\end{split}
\end{equation}
with $\mathbf E\|\delta_2\|=O(h^3)$ and $\mathbf{E}\|\delta_2\|^2=O(h^5).$ For  (\ref{3.8.11}), taking Taylor expansion to $P^h[1]$ and $Q^h[1]$ at $(p,q),$ we obtain
\begin{equation}
\begin{split}
\label{4.5}
P^h[1]=&p-f(q)h-vph-\frac{1}{2}\nabla^2F(q)Mph^2-\sum\limits_{r=1}^m\sigma_r\Delta_1W\\
&+\frac{1}{2}vf(q)h^2+\frac{1}{2}v^2ph^2+\frac{1}{2}v\sum\limits_{r=1}^m\sigma_rh\Delta_1W+\delta_3,
\end{split}
\end{equation}
\begin{equation}
\begin{split}
\label{4.6}
Q^h[1]=q+Mph-\frac{1}{2}f(q)h^2-\frac{1}{2}vMph^2-\frac{1}{2}\sum\limits_{r=1}^mM\sigma_rh\Delta_1W+\delta_4,
\end{split}
\end{equation}
where $\mathbf E\|\delta_i\|=O(h^3)$ and $\mathbf{E}\|\delta_i\|^2=O(h^5)$ with $i=3,4.$ Due to (\ref{4.3}) and (\ref{4.5}), we know
\begin{equation*}
\begin{split}
P(h)-P^h[1]=v\sum\limits_{r=1}^m\sigma_r\left(\int_{0 }^{h}\int_{0}^{s}dW_r(\theta)ds-\frac{1}{2}h\Delta_1W\right)+(\delta_1-\delta_3),
\end{split}
\end{equation*}
and thus $\|\mathbf{E}(P(h)-P^h[1])\|=O(h^3).$ Similarly, based on (\ref{4.4}) and (\ref{4.6}), we have $\|\mathbf E(Q(h)-Q^h[1])\|= O(h^3).$
For $i=2,3,4,5,$ it shows
\begin{align*}
&\left\|\mathbf{E}\left[(P(h)-p)^i-\mathbf (P^h[1]-p)^i\right]\right\|\leq Ch^3+O(h^4),\\
&\left\|\mathbf E\left[(Q(h)-q)^i-(Q^h[1]-q)^i\right]\right\|\leq Ch^3+O(h^4).
\end{align*}
Moreover, for $i_1+i_2=2,3,4,5$ and $i_1\geq 1$
\begin{align*}
\left\|\mathbf E\left[(P(h)-p)^{i_1}(Q(h)-q)^{i_2}-(P^h[1]-p)^{i_1}(Q^h[1]-q)^{i_2}\right]\right\|\leq Ch^3+O(h^4).
\end{align*}
By Taylor expansion and mean value theorem, we obtain
\begin{equation}
\begin{split}
\label{4.10}
&\left|\mathbf{E}\left[\psi(P(h),Q(h))-\psi(P^h[1],Q^h[1])\right]\right|\\
\leq&\left|\frac{\partial \psi}{\partial p}(p,q)\right|\left\|\mathbf{E}(P(h)-P^h[1])\right\|+\left|\frac{\partial \psi}{\partial q}(p,q)\right|\left\|\mathbf{E}(Q(h)-Q^h[1])\right\|\\
&+\sum\limits_{j=2}^{5}\sum\limits_{i=0}^j\left|\frac{\partial^j \psi(p,q)}{\partial p^i\partial q^{j-i}}\right|\left\|\mathbf{E}[(P(h)-p)^{i}(Q(h)-q)^{j-i}-(P^h[1]-p)^{i}(Q^h[1]-q)^{j-i}]\right\|\\
&+\sum\limits_{i=0}^{6}\mathbf{E}\left(\left|\frac{\partial^6 \psi(p+\theta_1P(h),q+\theta_1Q(h))}{\partial p^i\partial q^{6-i}}\right|\left\|(P(h)-p)^{i}(Q(h)-q)^{6-i}\right\|\right)\\
&+\sum\limits_{i=0}^{6}\mathbf{E}\left(\left|\frac{\partial^6 \psi(p+\theta_2P^h[1],q+\theta_2Q^h[1])}{\partial p^i\partial q^{6-i}}\right|\left\|(P^h[1]-p)^{i}(Q^h[1]-q)^{6-i}\right\|\right)
\end{split}
\end{equation}
with constants $0\leq\theta_1\leq 1$ and $0\leq\theta_2\leq 1.$ 
Here, based on \eqref{4.3}--\eqref{4.6}, Theorem \ref{th4.1.1} and Theorem \ref{th4.2}, we derive
\begin{align*}
&\mathbf{E}\left(\left|\frac{\partial^6\psi(p+\theta_1P(h),q+\theta_1Q(h))}{\partial p^i\partial q^{6-i}}\right|\left\|(P(h)-p)^{i}(Q(h)-q)^{6-i}\right\|\right)\\
\leq&C\left(\mathbf{E}\left\|(P(h)-p)^{2i}(Q(h)-q)^{12-2i}\right\|\right)^{\frac12}\leq Ch^{6-\frac{i}2},
\end{align*}
where we also use the fact
$\psi\in C^6_P(\mathbb{R}^{2d},\mathbb{R}),$ analogously,
\begin{align*}
&\mathbf{E}\left(\left|\frac{\partial^6 \psi(p+\theta_2P^h[1],q+\theta_2Q^h[1])}{\partial p^i\partial q^{6-i}}\right|\left\|(P^h[1]-p)^{i}(Q^h[1]-q)^{6-i}\right\|\right)=O(h^{6-\frac{i}2})
\end{align*}
for $0\leq i\leq 6.$
Finally, we deduce
\begin{equation}\label{localerror}
\left|\mathbf{E}\psi(P(h),Q(h))-\mathbf{E}\psi(P^h[1],Q^h[1])\right|\leq O(h^3),
\end{equation}
which, together with Theorem 9.1 in \cite{Milstein1988}, yields global weak order 2 for the proposed scheme (\ref{3.8.11}).
\end{proof}

According to above theorem and the condition \eqref{hypoelliptic}, we can get that the temporal average of the proposed scheme \eqref{3.8.11} is a proper approximation of the ergodic limit $\int_{\mathbb{R}^{2d}}\psi d\mu$. 
\begin{tm}\label{erlimit}
For any $\psi\in C_b^{6}(\mathbb{R}^{2d},\mathbb{R})$ and any initial values, under assumptions in Theorems \ref{ergodic} and \ref{th4.2}, the scheme \eqref{3.8.11} satisfies that
\begin{align*}
\left|\frac{1}{N}\sum_{n=1}^N\mathbf{E}\psi(P^h[n],Q^h[n])-\int_{\mathbb{R}^{2d}}\psi d\mu\right|\le C\left(h^2+\frac1T\right).
\end{align*}
\end{tm}

In fact, one can check that the assumptions in Theorem 5.6, \cite{MST10} are satisfied by \eqref{3.8.11}, and thus deduce this result.

\section{Numerical experiments}

The first example (Section \ref{sub5.1}) tests the numerical approximation by simulating a linear stochastic Langevin equation. In Section \ref{sub5.2}, numerical tests of conformal symplectic scheme for the nonlinear case are presented. In all the experiments, the expectation is approximated by taking average over 5000 realizations.

\subsection{A linear oscillator with damping}
\label{sub5.1}
Consider the following 2-dimensional stochastic Langevin equation
\begin{equation}\label{5.1}
\begin{split}
&dP=-aQdt-vPdt-\sigma dW(t),\quad P(0)=p,\\
&dQ=aPdt,\quad Q(0)=q,
\end{split}
\end{equation}
where $a,$ $v>0$ and $\sigma\neq 0$ are constants and $W(t)$ is a one-dimensional standard Wiener process. 
The solution to \eqref{5.1} possesses a unique invariant measure $\mu_1$:
\begin{align*}
d\mu_1=\rho_1(p,q)dpdq,
\end{align*}
where  $\rho_1(p,q)=\Theta\exp{\left(-\frac{av(p^2+q^2)}{\sigma^2}\right)}$
is known as the Boltzmann-Gibbs density and  $\Theta={\left(\int_{\mathbb{R}^2}\exp{\left(-\frac{av(p^2+q^2)}{\sigma^2}\right)}dpdq\right)}^{-1}$ is a renormalization constant.
The proposed scheme applied to \eqref{5.1} yields
\begin{equation}
\label{5.5}
\begin{split}
P_{n+1}=&e^{-vh}P_{n}-\frac{h^2}2a^2P_{n+1}-h\Big{(}1+\frac{vh}2\Big{)}e^{-vh}Q_{n}-\Big{(}1+\frac{vh}2\Big{)}e^{-vh}\sigma\Delta_{n+1}W,\\
Q_{n+1}=&Q_{n}+h\Big{(}1-\frac{vh}2\Big{)}e^{vh}aP_{n+1}+\frac{h^2}2a^2Q_{n}+\frac{h}2a\sigma\Delta_{n+1}W.
\end{split}
\end{equation}
We choose $p=3,$ $q=1,$ $v=2$ and $\sigma=0.5.$  Here we have taken the three different kinds of test functions (a) $\psi(P,Q)=\cos(P+Q),$ (b) $\psi(P,Q)=\exp{(-\frac{P^2}{2}-\frac{Q^2}{2})}$ and (c) $\psi(P,Q)=\sin(P^2+Q^2)$ as the test functions for weak convergence.

\begin{figure}[H]
\centering
\subfigure[$\psi(p,q)=\cos(p+q)$]{
\begin{minipage}{0.31\linewidth}
\includegraphics[width=4cm,height=4cm]{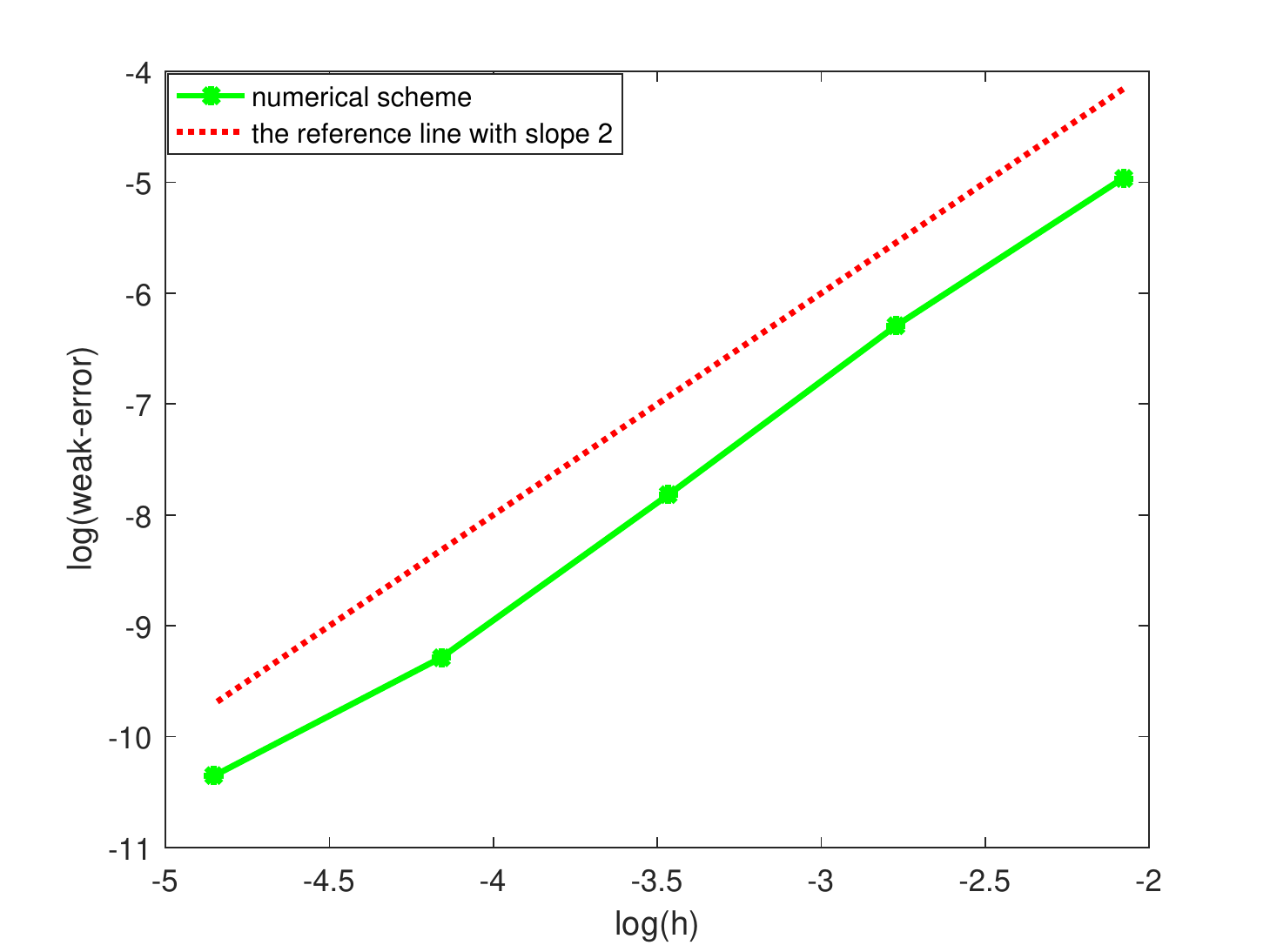}
\end{minipage}}
\subfigure[$\psi(p,q)=\exp(-\frac{p^2}2-\frac{q^2}2)$]{
\begin{minipage}{0.31\linewidth}
\includegraphics[width=4cm,height=4cm]{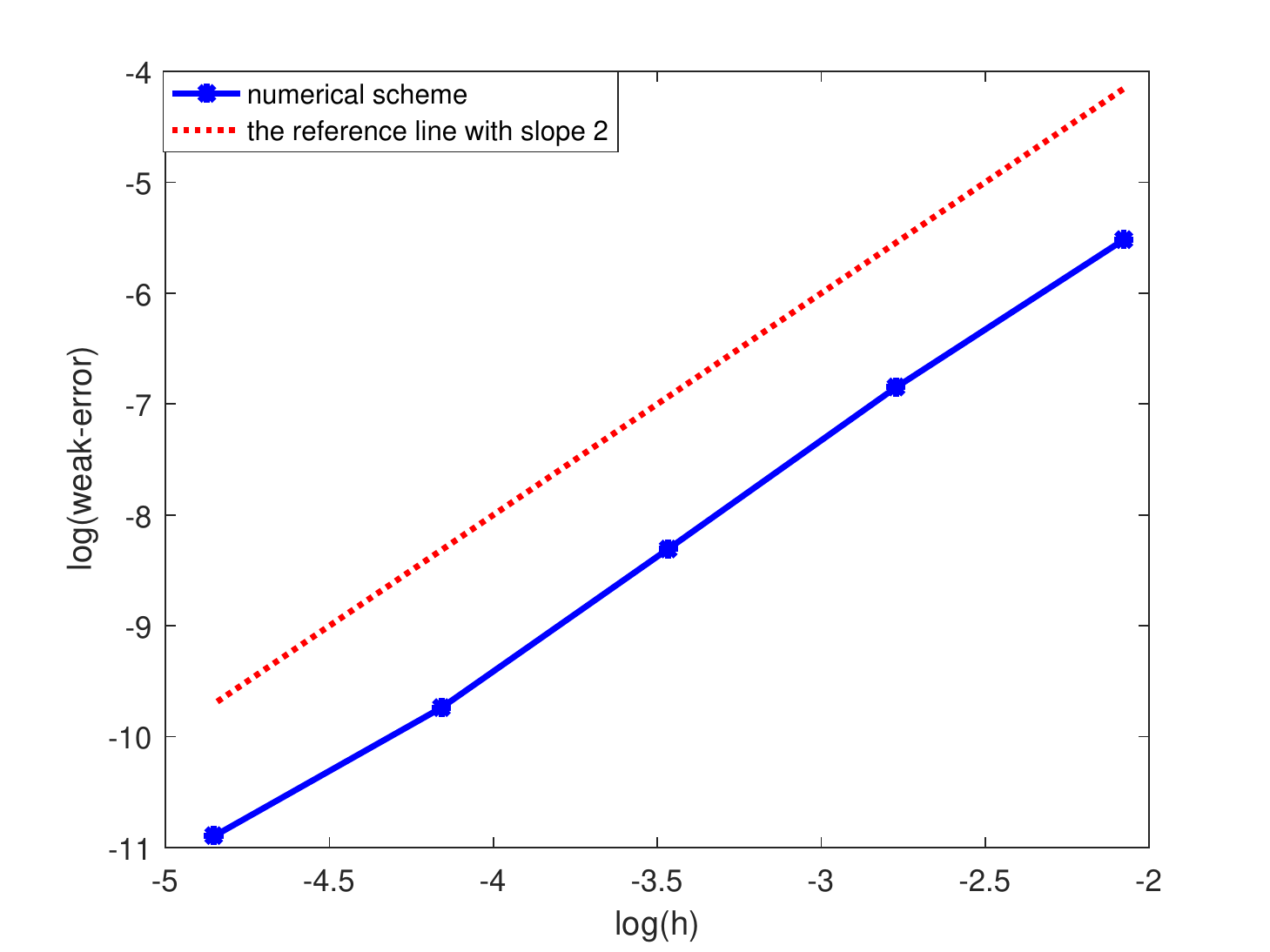}
\end{minipage}}
\subfigure[$\psi(p,q)=\sin(p^2+q^2)$]{
\begin{minipage}{0.31\linewidth}
\includegraphics[width=4cm,height=4cm]{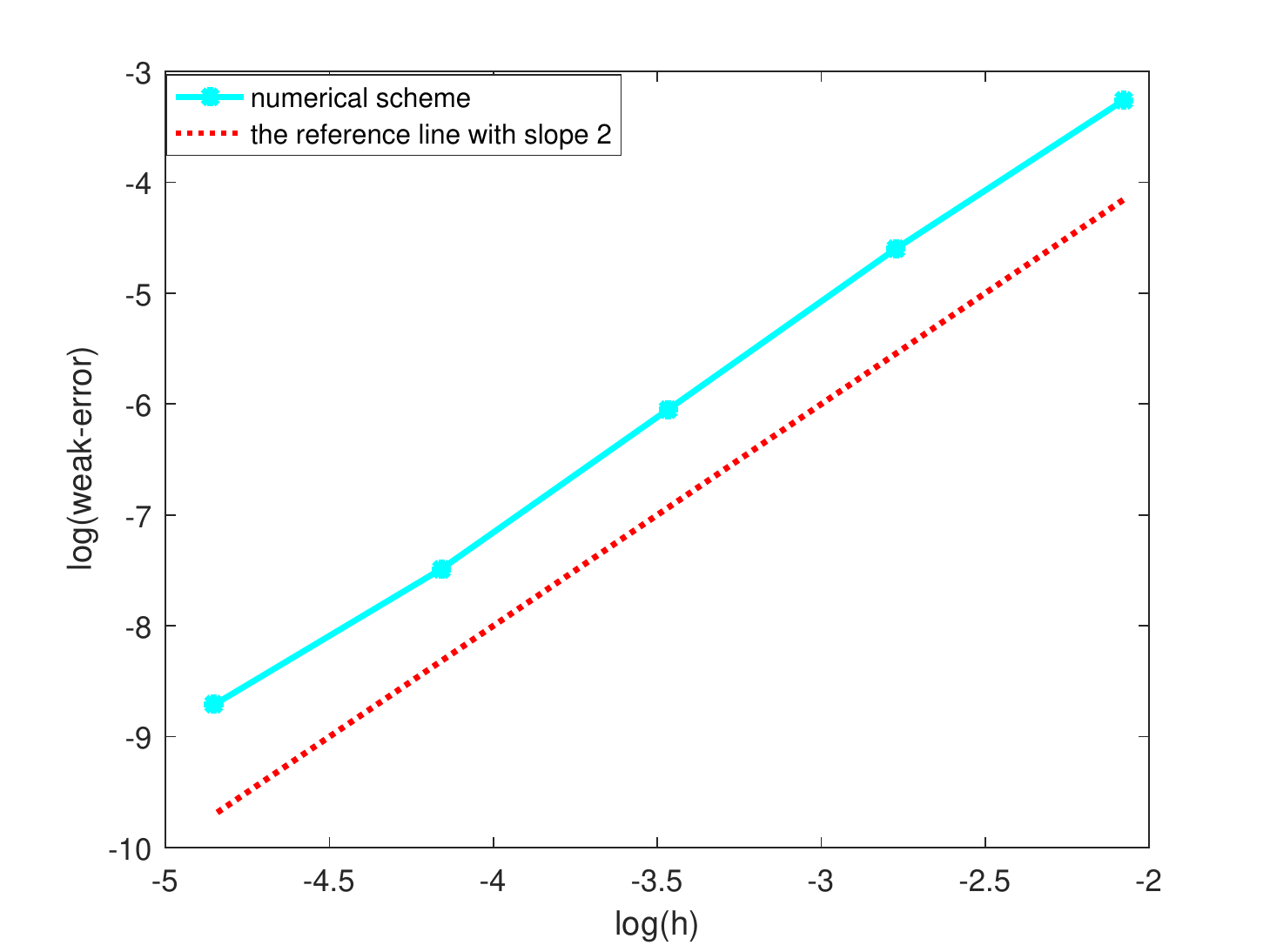}
\end{minipage}}
\caption{Rate of convergence in weak sense $(a=1,$ $v=2$ and $\sigma=0.5)$.}\label{pp1}
\end{figure}

Fig.\;\ref{pp1} plots the value $\ln|\mathbf{E}\psi (P(T),Q(T))-\mathbf{E}\psi(P_{N},Q_{N})|$ against $\ln h$ for five different step sizes $h=[2^{-3},2^{-4},2^{-5},2^{-6},2^{-7}]$ at $T=1$, where $(P(T),Q(T))$ and $(P_{N},Q_{N})$ represent the exact and numerical solutions at time $T$, respectively.  It can be seen that the weak order of (\ref{5.5}) is 2, which is indicated by the reference line of slope 2. 

\begin{figure}[H]
\centering
\subfigure[$\psi(p,q)=\cos(p+q)$]{
\begin{minipage}{0.31\linewidth}
\includegraphics[width=4cm,height=4cm]{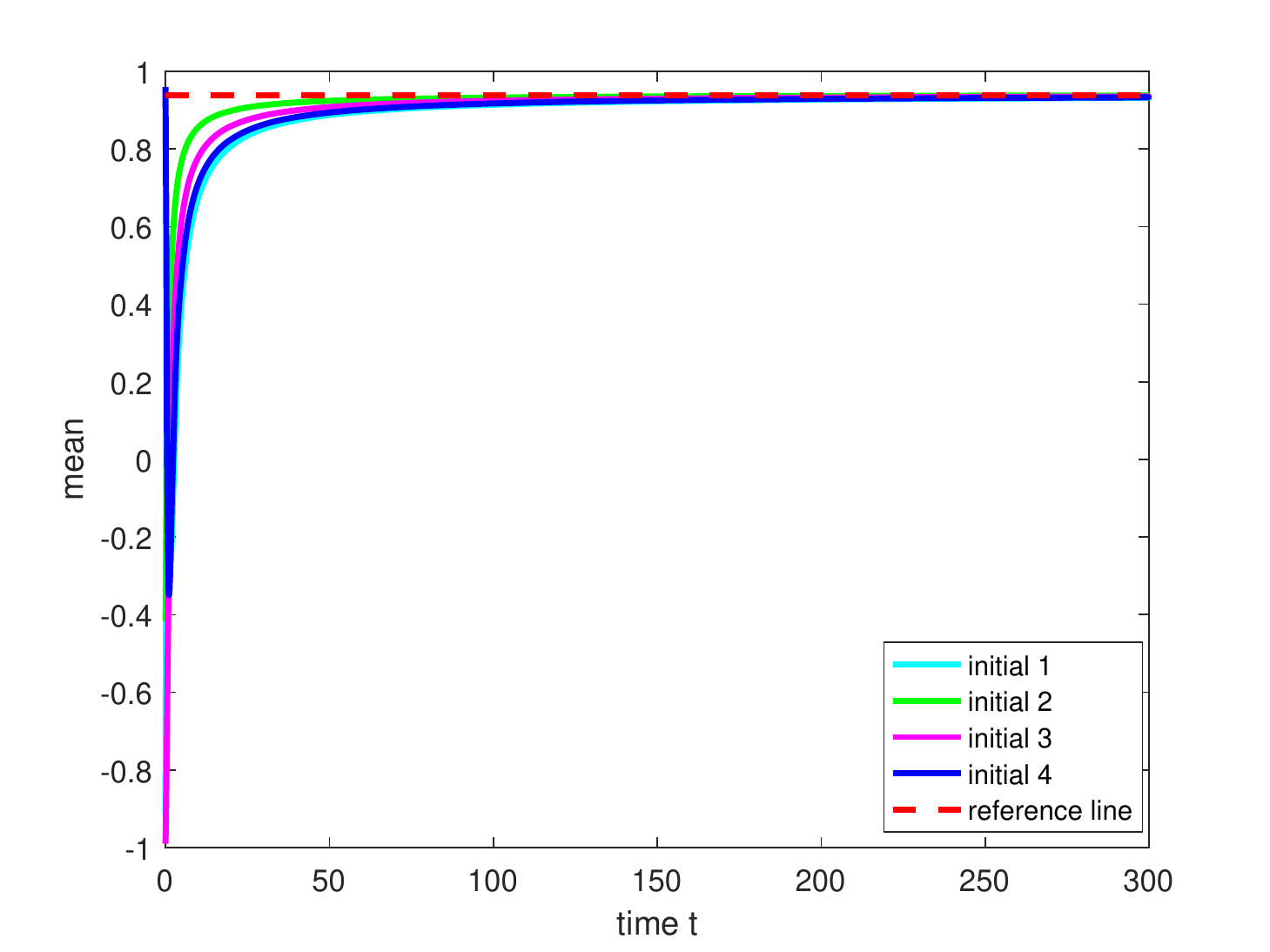}
\end{minipage}}
\subfigure[$\psi(p,q)=\exp(-\frac{p^2}2-\frac{q^2}2)$]{
\begin{minipage}{0.31\linewidth}
\includegraphics[width=4cm,height=4cm]{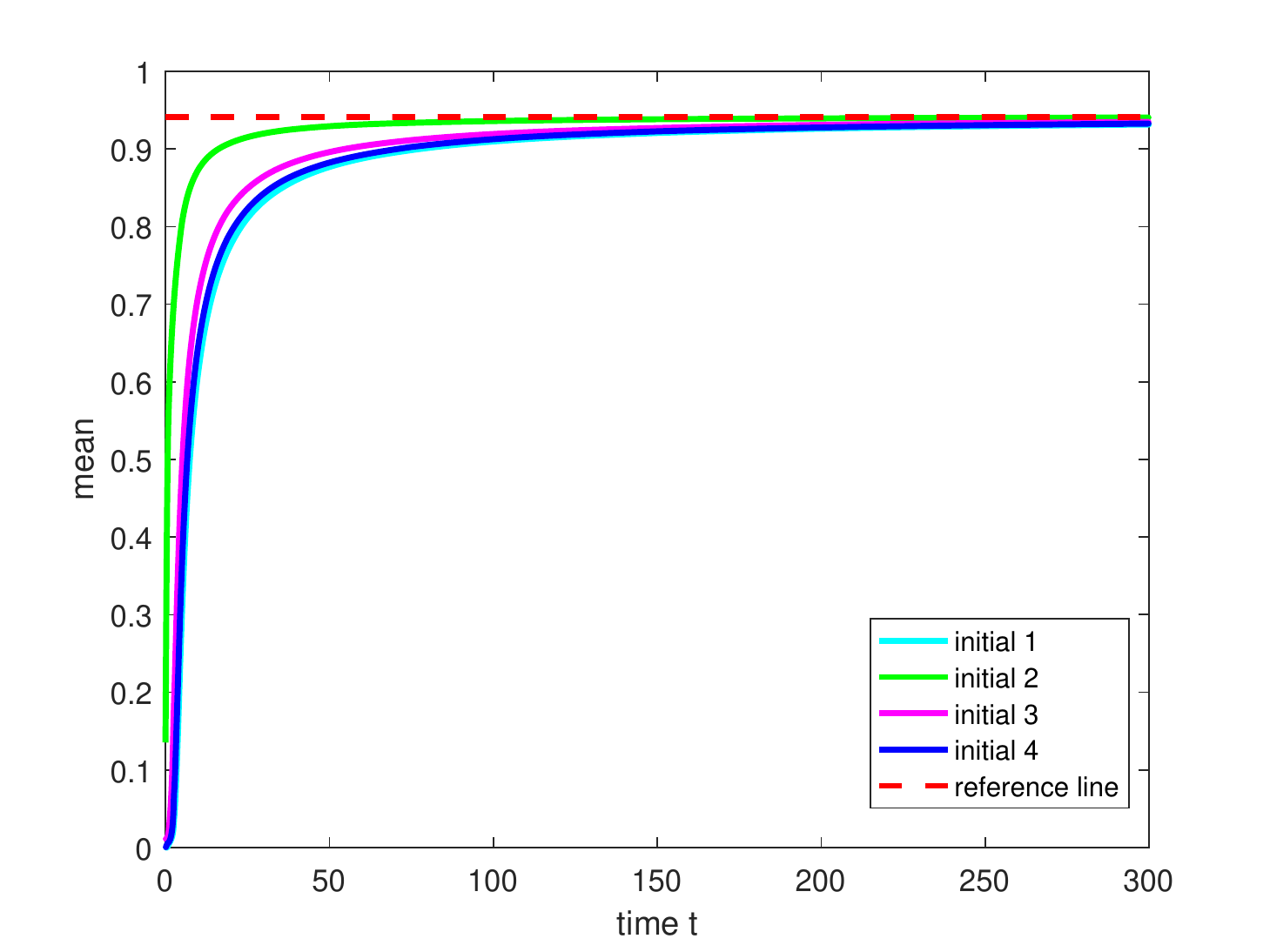}
\end{minipage}}
\subfigure[$\psi(p,q)=\sin(p^2+q^2)$]{
\begin{minipage}{0.31\linewidth}
\includegraphics[width=4cm,height=4cm]{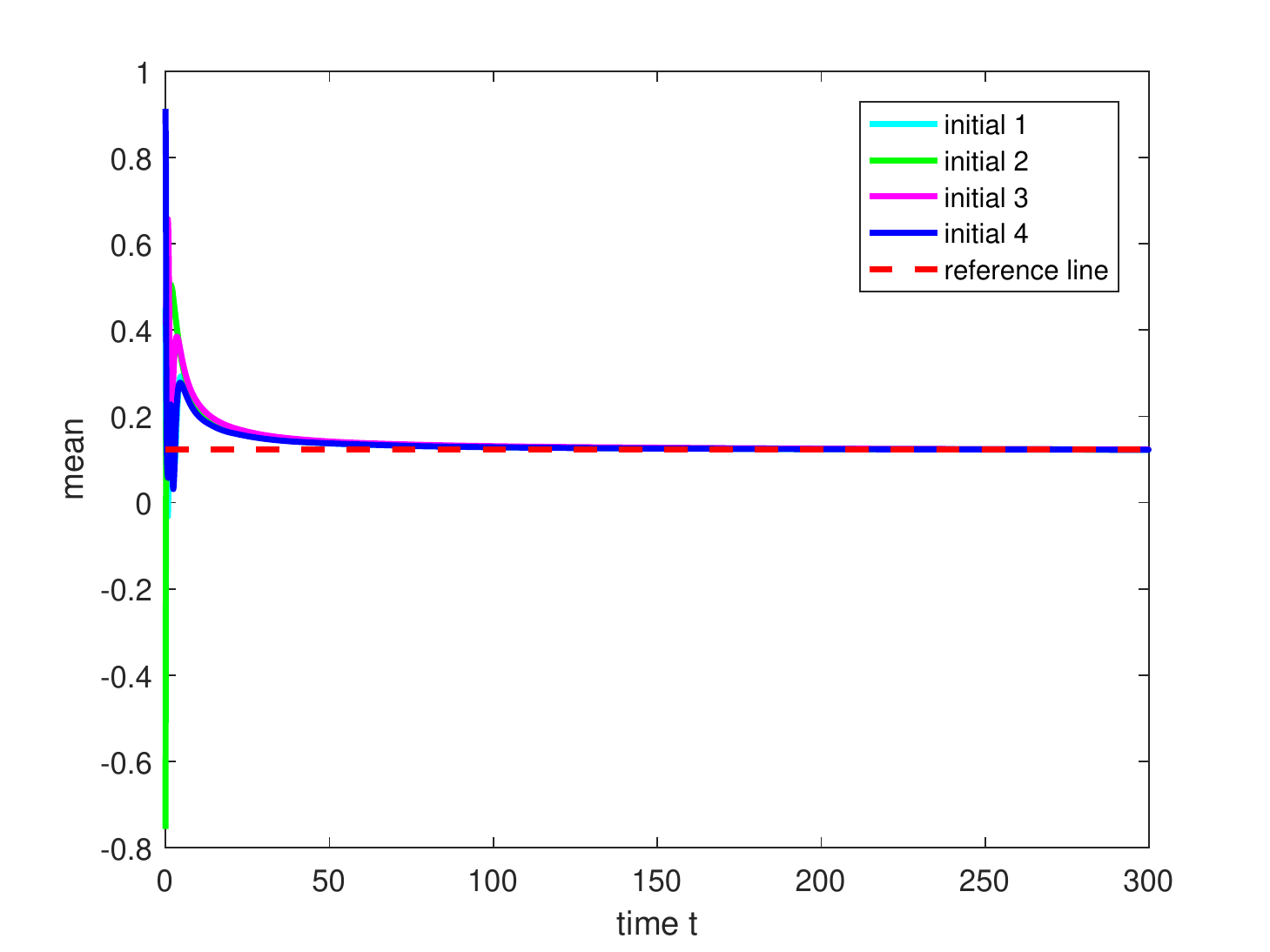}
\end{minipage}}
\caption{The temporal averages $\frac{1}{N}\sum_{i=1}^N\mathbf{E}\psi(P_i,Q_i)$ starting from different initial values $(a=1,$ $v=2,$ $\sigma=0.5$ and $T=300)$.}\label{p2}
\end{figure}

To verify that the temporal averages starting from different initial values will converge to the spatial average, i.e., the ergodic limit $$\int_{\mathbb{R}^2}\psi(p,q)d\mu_1=\int_{\mathbb{R}^2}\psi(p,q)\rho_1(p,q)dpdq,$$ we introduce the reference value for a specific test function $\psi$ to represent the ergodic limit: since the function $\psi$ is uniformly bounded and the density function $\rho_1$ dissipates exponentially, the integrator is almost zero when $p^2+q^2$ is sufficiently large. Thus, we choose $\int_{-10}^{10}\int_{-10}^{10}\psi(p,q)\rho_1(p,q)dpdq$ as the reference value, which appears as the dashed line in Fig.\;\ref{p2}. We can tell from Fig.\;\ref{p2} that the temporal averages $\frac{1}{N}\sum_{i=1}^N\mathbf{E}\psi(P_i,Q_i)$  of the proposed scheme starting from four different initial values initial$(1)=(-10,1)^\top,$ initial$(2)=(2,0)^\top,$
initial$(3)=(0,3)^\top$ and initial$(4)=(4,2)^\top$
converge to the reference line with error no more than $h^2+\frac1T$, which coincides with Theorem \ref{erlimit}.

\subsection{A nonlinear oscillator with linear damping} In this section, we consider the following equation
\label{sub5.2}
\begin{equation}\label{NSLE6.1}
\begin{split}
&dP=-(4Q^3-4Q-\frac12)dt-vPdt+\sqrt{2\beta^{-1}v}dW(t),\quad P(0)=p,\\
&dQ=Pdt,\quad Q(0)=q,
\end{split}
\end{equation}
where $v,$ $\beta>0$ are fixed constants and $W(t)$ denotes a one-dimensional standard Wiener process. 
Similar to \eqref{5.1}, \cite{MST10} shows that the dynamics generated by \eqref{NSLE6.1} is ergodic with the invariant measure $\mu_2,$ which can be characterized by the Boltzmann-Gibbs density
\begin{align*}
\rho_2(p,q)=\Theta\exp\bigg{(}-\beta\bigg{(}\frac12p^2+(1-q^2)^2-\frac{1}{2}q\bigg{)}\bigg{)}
\end{align*}
with the renormalization constant  $\Theta={\left(\int_{\mathbb{R}^2}e^{-\beta(\frac12p^2+(1-q^2)^2-\frac{1}{2}q)}dpdq\right)}^{-1}.$ Based on \eqref{3.8.11}, we get the associated numerical scheme
\begin{equation}\label{NSLENS}
\begin{split}
P_{n+1}=&e^{-vh}P_{n}-\frac{h^2}2P_{n+1}\big{(}12Q_{n}^2-4\big{)}-he^{-vh}\Big{(}1+\frac{vh}2\Big{)}\bigg{(}4Q_{n}^3-4Q_n-\frac{1}{2}\bigg{)}\\
&+e^{-vh}\Big{(}1+\frac{vh}2\Big{)}\sqrt{2\beta^{-1}v}\Delta_{n+1}W,\\
Q_{n+1}=&Q_{n}+he^{vh}\Big{(}1-\frac{vh}2\Big{)}P_{n+1}+\frac{h^2}2\bigg{(}4Q_{n}^3-4Q_n-\frac{1}{2}\bigg{)}-\frac{h}2\sqrt{2\beta^{-1}v}\Delta_{n+1}W.
\end{split}
\end{equation}
Although \eqref{NSLE6.1} does not satisfy the linear assumption in Theorem \ref{ergodic} and the Lipschitz assumption in Theorem \ref{th4.2}, we investigate its ergodicity and weak convergence order in the view of numerical tests.

Let $v=4,$ $\beta=2,$ and test functions $\psi$ be the same as those in Section \ref{sub5.1}.
The value $\ln|\mathbf{E}\psi (P(T),Q(T))-\mathbf{E}\psi(P_{N},Q_{N})|$ against $\ln h$ for five different step sizes $h=[2^{-3},2^{-4},2^{-5},2^{-6},2^{-7}]$ at $T=1$ is shown in Fig.\;\ref{pic3}, similar to Fig.\;\ref{pp1}. Compared with the reference lines of slope 2 in Fig.\;\ref{pic3}, it can be seen that (\ref{NSLENS}) has order $2$ in the sense of weak approximations.

\begin{figure}[H]
\centering
\subfigure[$\psi(p,q)=\cos(p+q)$]{
\begin{minipage}{0.31\linewidth}
\includegraphics[width=4cm,height=4cm]{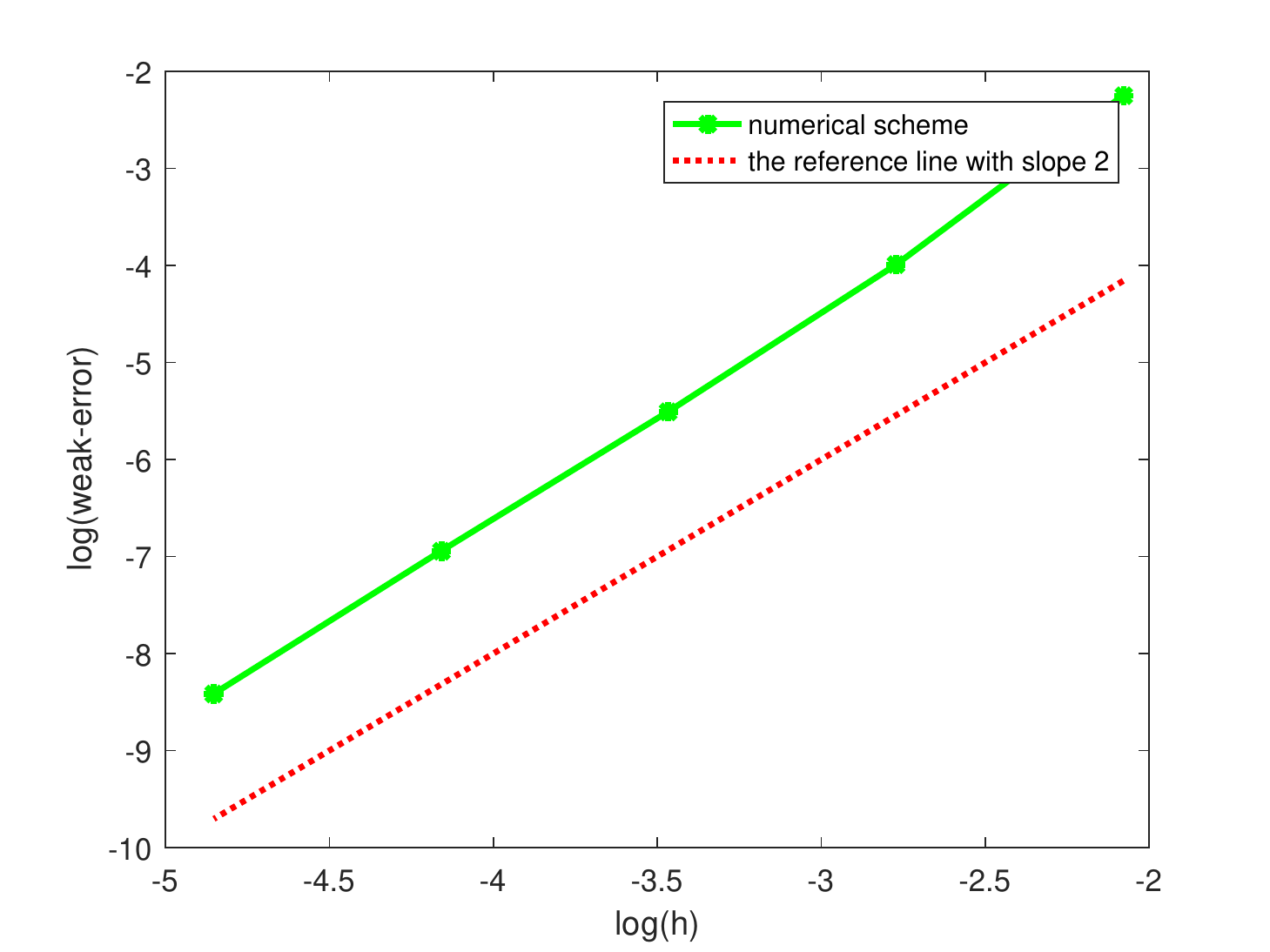}
\end{minipage}}
\subfigure[$\psi(p,q)=\exp(-\frac{p^2}2-\frac{q^2}2)$]{
\begin{minipage}{0.31\linewidth}
\includegraphics[width=4cm,height=4cm]{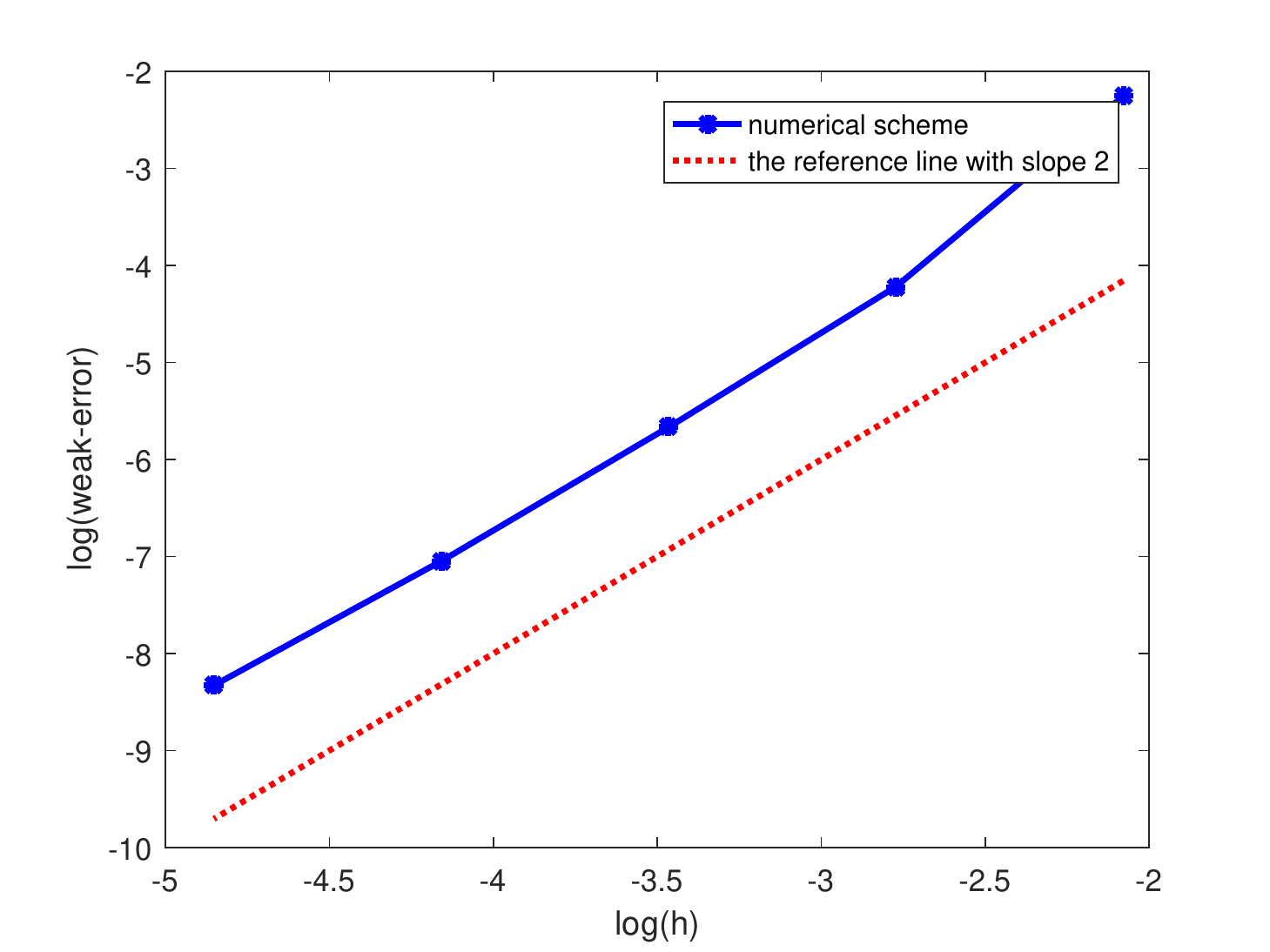}
\end{minipage}}
\subfigure[$\psi(p,q)=\sin(p^2+q^2)$]{
\begin{minipage}{0.31\linewidth}
\includegraphics[width=4cm,height=4cm]{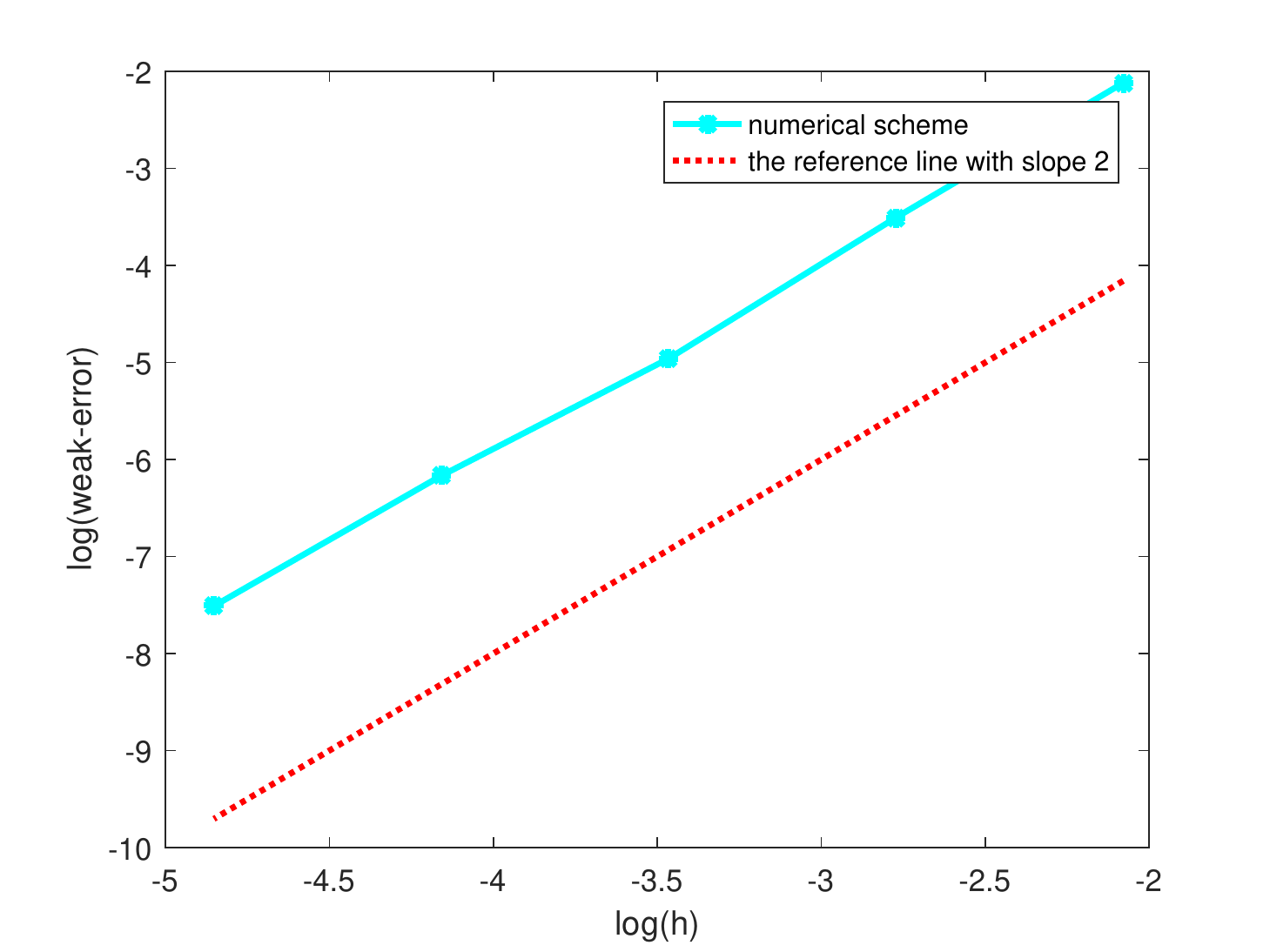}
\end{minipage}}
\caption{Rate of convergence in weak sense ($p=-2$ and $q=-2)$
.}\label{pic3}
\end{figure}

\begin{figure}[H]
\centering
\subfigure[$\psi(p,q)=\cos(p+q)$]{
\begin{minipage}{0.31\linewidth}
\includegraphics[width=4cm,height=4cm]{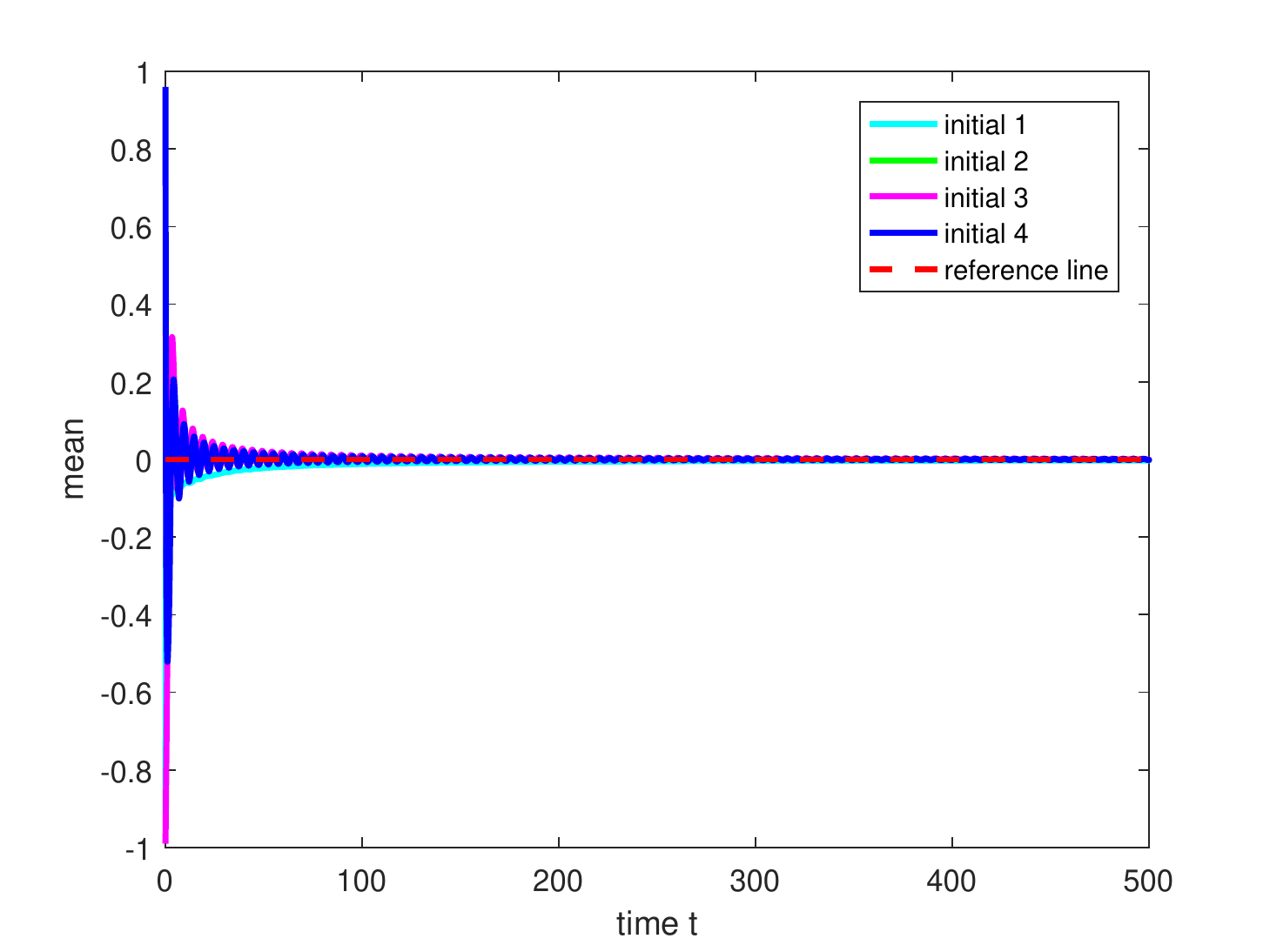}
\end{minipage}}
\subfigure[$\psi(p,q)=\exp(-\frac{p^2}2-\frac{q^2}2)$]{
\begin{minipage}{0.31\linewidth}
\includegraphics[width=4cm,height=4cm]{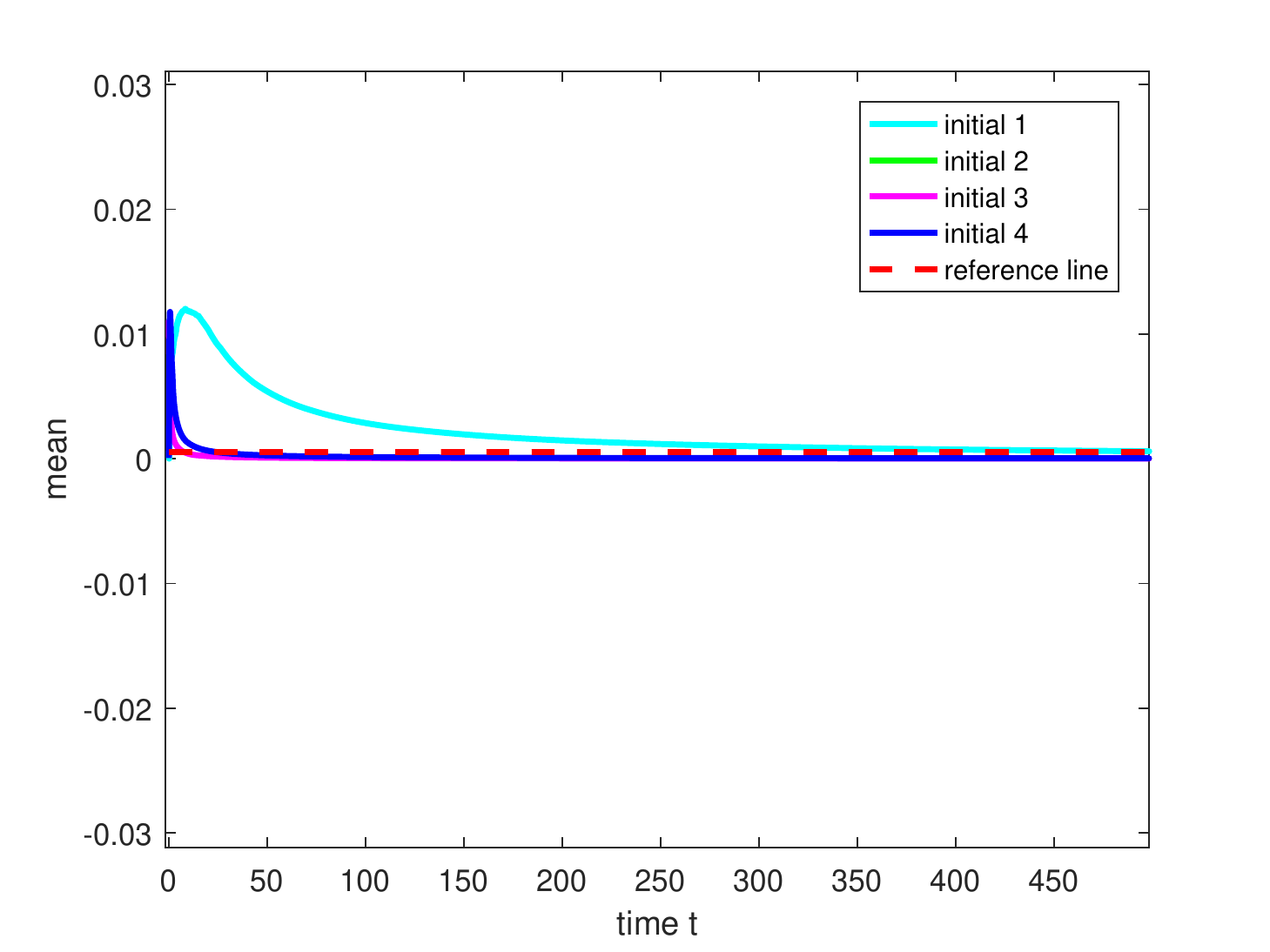}
\end{minipage}}
\subfigure[$\psi(p,q)=\sin(p^2+q^2)$]{
\begin{minipage}{0.31\linewidth}
\includegraphics[width=4cm,height=4cm]{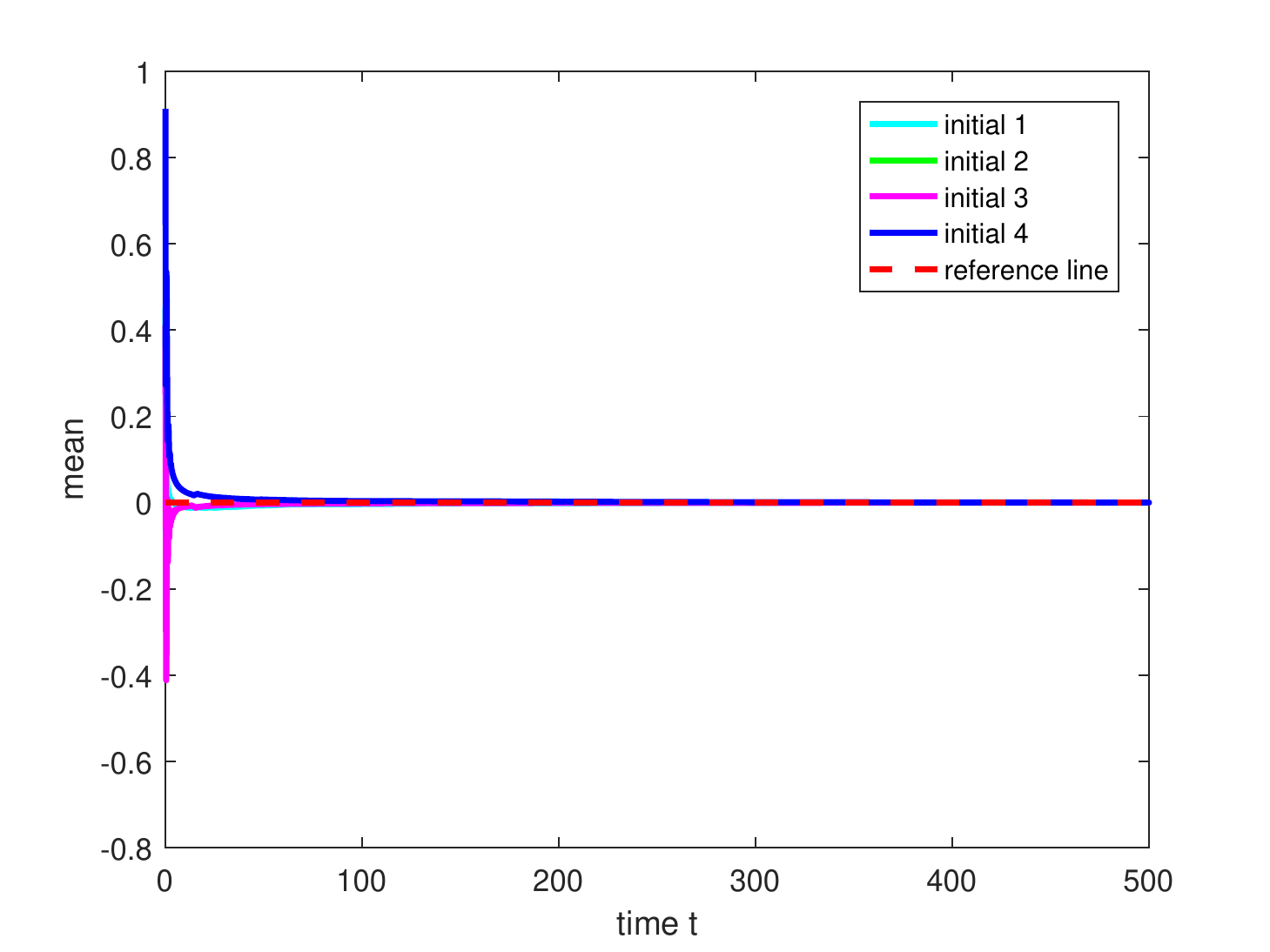}
\end{minipage}}
\caption{The temporal averages $\frac{1}{N}\sum_{i=1}^N\mathbf{E}\psi(P_i,Q_i)$ starting from different initial values with $T=500$
.}\label{pic4}
\end{figure}

Fig.\;\ref{pic4} shows the temporal averages $\frac{1}{N}\sum_{i=1}^N\mathbf{E}\psi(P_i,Q_i)$ of \eqref{NSLENS} starting from different initial values initial$(1)=(-10,1)^\top,$ initial$(2)=(2,0)^\top,$
initial$(3)=(0,3)^\top$ and initial$(4)=(4,2)^\top.$ We also use $\int_{-10}^{10}\int_{-10}^{10}\psi(p,q)\rho_2(p,q)dpdq$ as an approximation of the reference value, i.e., the ergodic limit $$\int_{\mathbb{R}^2}\psi(p,q)d\mu=\int_{\mathbb{R}^2}\psi(p,q)\rho_2(p,q)dpdq.$$ 
Fig.\;\ref{pic4} indicates that the proposed scheme is ergodic from the view of numerical tests.

\section{Conclusion}In this paper, an approach for constructing high weak order conformal symplectic schemes for stochastic Langevin equations is developed motivated by the ideas in \cite{Abdulle,WHS16,Anton1,MilTre2003}.
The key points are: the generating function is applied to ensure that the proposed scheme preserves the geometric structure, while the modified technique is used to reduce the simulation of multiple integrations. We show that, for the case $k=k'=1,$ the proposed scheme could inherit both the conformal symplectic geometric structure (under Lipschitz assumption) and the ergodicity (under linear assumption) of the stochastic Langevin equation. Numerical experiments verify our theoretical results. In addition, the numerical tests of an oscillator with non-global Lipschitz coefficients indicate that the proposed scheme could also inherit the internal properties of the original system, which implies that our results could possibly extend to the non-global Lipschitz case. The theoretical analysis of this extension is also ongoing.

\bibliography{references}
\bibliographystyle{plain}

\end{document}